\documentclass[12pt]{amsart}
\usepackage{amscd,amsmath,amsthm,amssymb}
\usepackage[left]{lineno}
\usepackage{color}
\usepackage{stmaryrd}
\usepackage[utf8]{inputenc}
\usepackage{cleveref}
\usepackage{epstopdf}
\usepackage{graphicx}
\usepackage{xcolor}
\usepackage{subeqnarray}
\usepackage{subcaption}
\usepackage{cases}

\usepackage{tikz}

\definecolor{verylight}{gray}{0.97}
\definecolor{light}{gray}{0.9}
\definecolor{medium}{gray}{0.85}
\definecolor{dark}{gray}{0.6}

\def\NZQ{\mathbb}               

\def\KK{{\NZQ K}}

\def\KK{{\NZQ K}}

\def\G{{\mathcal G}}

\def\0b{{\mathbf 0}}
\def\alphab{{\mathbf \alphab}}
\def\c_ib{{\mathbf c_i}}

\newcommand\calP{\mathcal{P}}

\def\reg{{\mathbf reg}}
\def\height{\operatorname{ht}}
\def\depth{\operatorname{depth}}
\def\opn#1#2{\def#1{\operatorname{#2}}} 
\opn\chara{char} \opn\length{\ell} \opn\pd{pd} \opn\rk{rk}
\opn\projdim{proj\,dim} \opn\injdim{inj\,dim} \opn\rank{rank}
\opn\depth{depth} \opn\grade{grade} \opn\height{height}
\opn\embdim{emb\,dim} \opn\codim{codim}

\opn\Tr{Tr} \opn\bigrank{big\,rank}
\opn\superheight{superheight}\opn\lcm{lcm}
\opn\trdeg{tr\,deg}
\opn\reg{reg} \opn\lreg{lreg} \opn\ini{in} \opn\lpd{lpd}
\opn\size{size} \opn\sdepth{sdepth}
\opn\link{link}\opn\fdepth{fdepth}\opn\lex{lex}
\opn\tr{tr}
\opn\type{type}
\opn\gap{gap}
\opn\arithdeg{arith-deg}
\opn\HS{HS}
\opn\GL{GL}
\opn\div{div} \opn\Div{Div} \opn\cl{cl} \opn\Cl{Cl}
\opn\Spec{Spec} \opn\Supp{Supp} \opn\supp{supp} \opn\Sing{Sing}
\opn\Ass{Ass} \opn\Min{Min}\opn\Mon{Mon}
\opn\Ann{Ann} \opn\Rad{Rad} \opn\Soc{Soc}\opn\Deg{Deg}
\opn\Im{Im} \opn\Ker{Ker} \opn\Coker{Coker} \opn\Am{Am}
\opn\Hom{Hom} \opn\Tor{Tor} \opn\Ext{Ext} \opn\End{End}
\opn\Aut{Aut} \opn\id{id}

\opn\nat{nat}
\opn\pff{pf}
\opn\Pf{Pf} \opn\GL{GL} \opn\SL{SL} \opn\mod{mod} \opn\ord{ord}
\opn\Gin{Gin} \opn\Hilb{Hilb}\opn\sort{sort}
\opn\PF{PF}\opn\Ap{Ap}
\opn\mult{mult}
\opn\bight{bight}
\opn\aff{aff}
\opn\relint{relint} \opn\st{st}
\opn\lk{lk} \opn\cn{cn} \opn\core{core} \opn\vol{vol}  \opn\inp{inp} \opn\nilpot{nilpot}
\opn\link{link} \opn\star{star}\opn\lex{lex}\opn\set{set}
\opn\width{wd}
\opn\Fr{F}
\opn\QF{QF}
\opn\G{G}
\opn\type{type}\opn\res{res}
\opn\conv{conv}
\opn\Ind{Ind}
\opn\gr{gr}

\def\pot#1#2{#1[\kern-0.28ex[#2]\kern-0.28ex]}

\opn\dirlim{\underrightarrow{\lim}}
\opn\inivlim{\underleftarrow{\lim}}
\let\to=\rightarrow

\def\Implies{\ifmmode\Longrightarrow \else
	\unskip${}\Longrightarrow{}$\ignorespaces\fi}
\def\implies{\ifmmode\Rightarrow \else
	\unskip${}\Rightarrow{}$\ignorespaces\fi}
\def\iff{\ifmmode\Longleftrightarrow \else
	\unskip${}\Longleftrightarrow{}$\ignorespaces\fi}

\let\:=\colon
\newtheorem{Theorem}{Theorem}[section]
\newtheorem{Lemma}[Theorem]{Lemma}
\newtheorem{Corollary}[Theorem]{Corollary}

\newtheorem{Remark}[Theorem]{Remark}

\newtheorem{Example}[Theorem]{Example}

\newtheorem{Definition}[Theorem]{Definition}

\newtheorem{Notation}[Theorem]{Notation}

\let\epsilon\varepsilon
\let\kappa=\varkappa
\def\qed{\ifhmode\textqed\fi
	\ifmmode\ifinner\quad\qedsymbol\else\dispqed\fi\fi}
\def\textqed{\unskip\nobreak\penalty50
	\hskip2em\hbox{}\nobreak\hfil\qedsymbol
	\parfillskip=0pt \finalhyphendemerits=0}
\def\dispqed{\rlap{\qquad\qedsymbol}}

\opn\dis{dis}
\def\pnt{{\raise0.5mm\hbox{\large\bf.}}}

\opn\Lex{Lex}

\newcommand*{\circled}[1]{\lower.7ex\hbox{\tikz\draw (0pt, 0pt)%
		circle (.5em) node {\makebox[1em][c]{\small #1}};}}

\begin{document}

\title{Powers of edge ideals of edge-weighted trees}

	\author {Jiaxin Li,  Guangjun Zhu$^{\ast}$ and  Shiya Duan}

	
	\address{School of Mathematical Sciences, Soochow University, Suzhou, Jiangsu, 215006, P. R. China}
	
\email{lijiaxinworking@163.com(Jiaxin Li), zhuguangjun@suda.edu.cn\linebreak[4](Corresponding author:Guangjun Zhu),3136566920@qq.com(Shiya Duan).}

	
	\thanks{$^{\ast}$ Corresponding author}



	

	\maketitle
\begin{abstract}
This paper gives exact formulas for the regularity of edge ideals of  edge-weighted integrally closed trees. In addition, we provide some linear upper bounds  on the regularity of powers of such  ideals.
\end{abstract}
	\setcounter{tocdepth}{1}

    \section{Introduction}
 Let $G$ be a graph with vertex set $V(G)=\{x_1,\ldots,x_n\}$ and   edge set $E(G)$.  We write $xy$ for $\{x,y\}$ if $\{x,y\}\in E(G)$  is an  edge of $G$ with $x$ and $y$ as endpoints.
Suppose $w: E(G)\rightarrow \mathbb{Z}_{>0}$ is an edge weight function on $G$. We write $G_\omega$ for the pair $(G,\omega)$ and call it an {\em edge-weighted} graph with the underlying graph  $G$.
For a weighted graph $G_\omega$, its {\em edge-weighted ideal}  (or simply edge ideal), was introduced in \cite{PS}, is the ideal of the polynomial ring $S=\KK[x_{1},\dots, x_{n}]$ in $n$ variables over a field $\KK$ given by
\[
I(G_\omega)=(x^{\omega(xy)}y^{\omega(xy)}\mid xy\in E(G_\omega)).
\]
If $w$ is the constant function defined by $w(e)=1$ for all $e\in E(G)$, then  $I(G_\omega)$ is the classical  edge ideal of the underlying graph $G$ of $G_\omega$, which
has been  studied extensively in the literature \cite{BBH,BHT,JNS,M,Zhu,Zhu1}.

Recently, there has been some interest in characterizing weights for which the edge ideals of edge-weighted graphs are Cohen-Macaulay. For example,
Paulsen and Sather-Wagstaff in \cite{PS}  classified Cohen-Macaulay edge-weighted graphs $G_\omega$ where the underlying
graph $G$ is a  cycle, a tree, or a  complete graph.  Seyed Fakhari et al. in \cite{FSTY} continued this study by classifying
Cohen-Macaulay edge-weighted graphs $G_\omega$ if $G$ is a very well-covered graph. Recently, Diem et al. in \cite{DMV} gave a  complete  characterization of sequentially Cohen-Macaulay edge-weighted graphs. In \cite{W},  Wei classified all Cohen-Macaulay weighted chordal graphs from a purely graph-theoretic point of view. Hien in \cite{Hi}
classified Cohen-Macaulay edge-weighted graphs $G_\omega$ if $G$ has girth at least $5$.

 Integral closure of monomial ideals is also an interesting topic.
In \cite{DZCL}, we   gave a  complete  characterization of integrally closed edge-weighted graphs $G_\omega$ and  showed that if their underlying
graph $G$ is a  star graph,  a  path, or a cycle, then $G_\omega$ is normal. Later,  in \cite{ZDCL}, we gave some exact formulas for the regularity  of powers of edge ideals of  edge-weighted star graphs and integrally closed paths.

The study of edge ideals of edge-weighted graphs is much more recent and consequently there are  fewer results in this direction.
In this paper, we decide to focus on the regularity  of powers of the edge ideals  of   integrally closed edge-weighted trees. Recall that the regularity is an important invariant associated to a homogeneous ideal  $I$.
It is well-known that   $\reg(I^t)$, as a function in $t$, is asymptotically linear for $t\gg 0$  (cf.~\cite{CHT,K}). In general, it is very difficult to decide when this function starts to be linear.
To find the exact form of the linear function is also not easy (cf.~\cite{BBH,EU,H1,ZXWZ}).

This paper is organized as follows. In the next section, we recall several definitions and
terminology that we will need later. In Section \ref{sec:regtree},  by  using the Betti splitting and polarization approaches, we give precise formulas for the  regularity of the edge ideals of  edge-weighted integrally closed trees. In Section \ref{sec:powertree}, by classifying based on  the distance from a vertex  of the tree to  its longest path   containing all non-trivial edges, we give some linear upper bounds on the regularity of powers of  edge ideals of  edge-weighted integrally closed trees.

\section{Preliminaries}
\label{sec:prelim}

In this section, we provide the definitions and basic facts that will be used throughout this paper. For detailed information we refer to \cite{HH}.

\subsection{Notions of simple graphs}
Throughout the paper, all graphs will be finite and simple, i.e., undirected graphs with no loops nor multiple edges. Given an edge-weighted graph $G_\omega$,  we denote its vertex and edge sets by $V(G_\omega)$ and $E(G_\omega)$, respectively. Any concept valid for graphs automatically applies to edge-weighted graphs.
 For example, the \emph{neighborhood} of a vertex $v$ in an edge-weighted graph $G_\omega$ with the underlying graph  $G$ is defined as  $N_G(v)\!:=\{u \in V(G)\mid uv\in E(G)\}$. Given  a subset $W$ of $V(G_\omega)$, its \emph{neighborhood} is defined as  $N_G(W)\!:=\bigcup\limits_{v\in W} N_G(v)$. The \emph{induced subgraph} by $W$ in $G_\omega$ is the graph $G_\omega[W]$ with vertex set $W$, and for any  $u,v\in V(G_\omega[W])$,
 $uv$ is an  edge in $G_\omega[W]$ if and only if $uv$ is an  edge in $G_\omega$, and the weight function $\omega'$ satisfies $\omega'(uv)=\omega(uv)$.
 At the same time, the induced subgraph of $G_\omega$ on the set $V(G_\omega)\setminus W$ is denoted by $G_\omega\setminus W$ or $G\setminus W$ for simplicity. In particular, if $W=\{v\}$, we  write $G\setminus v$ instead of $G\setminus \{v\}$ for simplicity.
For any subset $A$ of $E(G_\omega)$, $G_\omega\setminus A$  is a subgraph of $G_\omega$ obtained by removing all edges in $A$. In particular, if $A=\{e\}$  then we also write $G_\omega\setminus e$ or $G\setminus e$ instead of $G_\omega\setminus \{e\}$.

A \emph{walk} $W$ of length $n$ in a graph $G$ is a sequence of vertices $(w_1,\ldots, w_n,w_{n+1})$,
such that $w_iw_{i+1}\in E(G)$ for $1\le i\le n$. The vertices $w_1$ and $w_{n+1}$ are connected  by $W$ and are called its ends, the vertices $w_2,\ldots,w_n$ are the inner vertices of $W$. The walk $W$ is \emph{closed} if $w_1=w_{n+1}$.
Furthermore, the walk $W$ is called a \emph{cycle} if it is closed and the points $w_1,\ldots,w_n$ are distinct. At the same time, a \emph{path}  denoted by $P$  is a walk where all  points are distinct.
  A \emph{tree} is a connected simple graph without cycles.
For a tree $T$, let $L(T)=\{v \in V(T)|\deg_T(v) = 1\}$ be the set of all leaves of $T$.  The tree $T$ is called to be a \emph{caterpillar} if $T\backslash L(T)$ is either empty or is a simple path.  A longest path in a caterpillar is called  the \emph{spine} of the caterpillar. Note that given
any spine, every edge of a caterpillar is incident on it. With respect to a fixed spine $P$, the the
pendant edges incident on $P$ are called \emph{whiskers}.

A matching in a graph is a subset of edges no two of which share of vertex. A matching is induced if no two vertices belonging to different edges of the matching are adjacent. In other words, an induced matching in a graph $G$ is formed by the edges of a 1-regular induced subgraph of $G$. If $G_\omega$ is an edge-weighted graph, its induced matching refers to the induced matching of its underlying graph  $G$.
The \emph{induced matching number} of $G_\omega$, denoted by $\nu(G_\omega)$ or $\nu(G)$, is the maximum size of an induced matching in $G$.

An edge-weighted  graph is said to be  \emph{non-trivial} if there is at least one edge with  a weight  greater than $1$, otherwise, it is said to be
{\em trivial}. An edge $e \in E(G_\omega)$ is said to be {\em non-trivial} if its weight $w(e) \ge2$. Otherwise, it is said to be {\em trivial}.

\subsection{Notions from commutative algebra}

Let  $S=\KK[x_1,\ldots,x_n]$ be a polynomial ring   over  a field $\KK$. Let  $M$ be a graded $S$-module with minimal free resolution
$$0\rightarrow \bigoplus\limits_{j}S(-j)^{\beta_{p,j}}\rightarrow \bigoplus\limits_{j}S(-j)^{\beta_{p-1,j}}\rightarrow \cdots\rightarrow \bigoplus\limits_{j}S(-j)^{\beta_{0,j}}\rightarrow M\rightarrow 0,$$
where the maps are exact, $p\le n$, and $S(-j)$ is the free module obtained by  shifting the degrees in $S$ by $j$.
The numbers $\beta_{i,j}$'s are positive integers and are called   the $(i, j)$-th graded Betti number of $M$.
An important homological invariant   related to these numbers are  the Castelnuovo-Mumford regularity (or simply regularity), denoted by $\reg(M)$,
	\begin{align*}
\reg(M)&=\mbox{max}\,\{j-i\ |\ \Tor_i(M,\KK)_j\neq 0\}\\
	\end{align*}

The following lemmas are often used  to compute  the  regularity of a module or  ideal.
 \begin{Lemma}{\em (\cite[Lemma 3.1]{HT})}\label{reglemma}
 Let $0 \longrightarrow M \longrightarrow N \longrightarrow P \longrightarrow 0$ be a short exact sequence of finitely generated graded $S$-modules. Then
 \[
 \reg(N) \leq \max \{\reg(M), \reg(P)\}.
 \]
 The equality holds if $\reg(P) \neq \reg(M)-1$.
 \end{Lemma}

\begin{Lemma}{\em (\cite[Lemma 3.1]{Zhu}) and \cite[Lemma 3.1]{HT}) }\label{sum}
 Let $S_1=k\left[x_1, \ldots, x_m\right]$ and $S_2=k\left[x_{m+1}, \ldots, x_n\right]$ be two polynomial rings, $I \subseteq S_1$ and $J \subseteq S_2$ be two nonzero homogeneous ideals. Then
\begin{itemize}			
\item[(1)] $\reg(I J)=\reg(I)+\reg(J)$,
\item[(2)] $\reg((I+J)^t)=\max _{\substack{i \in [t-1] \\ j \in [t]}}\{\reg(I^{t-i})+\reg(J^i), \reg(I^{t-j+1})+\reg(J^j)-1\}$ for any $t \ge 1$.
\end{itemize}
\end{Lemma}

Calculating or even estimating the regularity for a general ideal is a challenging problem. Formulas for $\reg(I)$ in special cases will be provided using methods developed in
\cite{FHT} and \cite{F}. For  a monomial ideal $I$,   let $\mathcal{G}(I)$ denote its unique minimal set of monomial generators.

\begin{Definition}{\em (\cite[Definition 1.1]{FHT})}\label{bettispliting}
		Let $I$ be a monomial ideal. If there exist monomial ideals $J$ and $K$ such that $\mathcal{G}(I)=\mathcal{G}(J)\cup\mathcal{G}(K)$ and $\mathcal{G}(J)\cap\mathcal{G}(K)=\emptyset$. Then $I=J+K$ is a Betti splitting if
		$$
		\beta_{i, j}(I)=\beta_{i, j}(J)+\beta_{i, j}(K)+\beta_{i-1, j}(J \cap K) \text { for all } i, j \geq 0,
		$$
		where $\beta_{i-1, j}(J \cap K)=0$ for $i=0$.
\end{Definition}

\begin{Lemma}{\em (\cite[Corollary 2.7]{FHT})}\label{spliting}
Suppose  $I=J+K$, where $\mathcal{G}(J)$ contains all the generators of $I$ that are divisible by some variable $x_i$, and $\mathcal{G}(K)$ is a nonempty set containing the remaining generators of $I$. If $J$ has a linear resolution, then $I=J+K$ is a Betti splitting.

Definition \ref{bettispliting} states that $\reg(I)=\max\{\reg(J), \reg(K), \reg(J \cap K)-1\}$,  as a result of Betti splitting.
\end{Lemma}

\begin{Definition}  {\em (\cite[Definition 2.1]{F})}\label{polarization}
Let $I\subset S$ be a monomial ideal with $\mathcal{G}(I)=\{u_1,\ldots,u_m\}$ where $u_i=\prod\limits_{j=1}^n x_j^{a_{ij}}$ for $i=1,\ldots,m$.
The polarization of $I$, denoted by $I^{\calP}$, is a squarefree monomial ideal in the polynomial ring $S^{\calP}$
$$I^{\calP}=(\calP(u_1),\ldots,\calP(u_m))$$
where $\calP(u_i)=\prod\limits_{j=1}^n \prod\limits_{k=1}^{a_{ij}} x_{jk}$ is a squarefree monomial  in $S^{\calP}=\KK[x_{j1},\ldots,x_{ja_j}\mid j=1,\ldots,n]$ and $a_j=\max\{a_{ij}| i=1,\ldots,m\}$ for  $1\leq j\leq n$.
\end{Definition}

A monomial ideal  and its polarization  share many homological and
algebraic properties.  The following is a  useful property of the polarization.

\begin{Lemma}{\em (\cite[Corollary 1.6.3]{HH})}\label{polar}
 Let $I\subset S$ be a monomial ideal and $I^{\calP}\subset S^{\calP}$ be its polarization.
Then
\[
\beta_{ij}(I)=\beta_{ij}(I^{\calP})
\]
 for all $i$ and $j$. In particular, $\reg(I)=\reg(I^{\calP})$.
\end{Lemma}

\medskip

\section{Regularity of the edge ideal of an edge-weighted integrally closed tree }
\label{sec:regtree}

In this section, we will give precise formulas for the  regularity of the edge ideals of  edge-weighted integrally closed trees. We first recall the definition of the integral closure of an ideal.

\begin{Definition}{\em (\cite[Definition 1.4.1]{HH})}\label{integrallyclosed}
 Let  $I$  be an ideal in a ring $R$.
    An element $f \in R$ is  said to be {\em integral} over $I$ if there exists an equation
    \[
    f^k+c_1f^{k-1}+\dots+c_{k-1}f+c_k=0 \text{\ \ with\ \ }c_i \in I^i.
    \]
   The set $\overline{I}$ of elements in $R$ which are integral over $I$  is the \emph{integral closure} of $I$.
If $I=\overline{I}$, then $I$  is said to be  {\em integrally closed}.
    An edge-weighted graph $G_\omega$ is said to be  {\em integrally closed} if its edge ideal $I(G_\omega)$ is integrally closed.
    \end{Definition}

According to \cite[Theorem 1.4.6]{HH}, every edge-weighted graph $G_\omega$ with trivial weights is integrally closed. The following lemma gives a complete characterization of a non-trivial edge-weighted graph that is integrally closed.

 \begin{Lemma}{\em (\cite[Theorem 3.6]{DZCL})}\label{integral}
 If  $G_\omega$ is a  non-trivial edge-weighted graph, then $I(G_\omega)$ is  integrally closed if and only if  $G_\omega$  does not contain  any of the following three graphs as induced subgraphs.
    \begin{enumerate}
    \item  A path $P_\omega$ of length $2$ where  all edges have non-trivial weights.
    \item The  disjoint union $P_\omega\sqcup P_\omega$ of two paths $P_\omega$  of length $1$ where all edges have non-trivial weights.
    \item A $3$-cycle $C_\omega$ where all edges have non-trivial weights.
\end{enumerate}
\end{Lemma}

From the lemma above, we can derive
\begin{Corollary} \label{integ}
Let $P_{\omega}$ be a  non-trivial  integrally closed path  with $n$ vertices, then  it can have  at most two edges with non-trivial weights.
\end{Corollary}

For a trivially weighted tree, we have
\begin{Lemma}{\em (\cite[Theorem 4.7]{BHT})}\label{treetri}
If $G_\omega$ is a trivial  weighted tree, then
 \[
 \reg(S/I(G_{\omega})^t)=2t+\nu(G)-1
 \]
for all $t \geq 1$, where $\nu(G)$ is the induced matching number of  the underlying graph $G$ of $G_\omega$.
\end{Lemma}

Thus, we will now consider a non-trivial edge-weighted integrally closed tree which satisfies the following conditions.

\begin{Remark}\label{setup}
Let $G_\omega$ be a non-trivial edge-weighted integrally closed tree, and let $P_\omega$ be the longest path containing all non-trivial edges in $G_\omega$. Then $P_\omega$ is an induced subgraph of $G_\omega$. Suppose $V(P_\omega)=\{x_1,\ldots,x_k\}$, $\omega_i=\max\{\omega_t\mid\omega_t=\omega(e_t)\text{\ and }  e_t=x_tx_{t+1} \text{\ for each } 1\le t\le k-1\}$ with  $\omega_i\ge 2$ and $\omega_i\ge \omega_{i+2}$ if $e_{i+2}\in E({P_{\omega}})$ for simplicity. In this case, $G_\omega \backslash e_i$ is  the disjoint union of the two trees $G_{\omega}^{1}$ and $G_{\omega}^{2}$ with $x_{i} \in V({G_{\omega}^{1}})$ and $x_{i+1} \in V({G_{\omega}^{2}})$.

Furthermore, let $s_i(G_\omega)=\max \{|M|\!: M \text{\ is an induced matching of\ }G_\omega \text{\ containing}$
 $\text{the edge\ } e_i\}$,  $s_{i+2}(G_\omega)=\max \{|M|\!: M \text{\ is an induced matching of\ }G_\omega \text{\ containing}$
 $\text{the edge\ } e_{i+2}\}$ if $e_{i+2}\in E({G_{\omega}})$
 and $s_{i+2}(G_\omega^{2})=\max \{|M|\!: M \text{\ is an induced matching}$ $\text{of\ }G_\omega^{2} \text{\ containing the edge\ } e_{i+2}\}$  if $e_{i+2}\in E({G_{\omega}^{2}})$.
\end{Remark}

\begin{Lemma}\label{G3}
Let $G_\omega$ be a non-trivial integrally closed tree as in Remark \ref{setup}. Let $A=\{x_i,x_{i+1}\}$ and  $G_{\omega}^{3}=G_\omega \backslash N_G(A)$, then
$\nu(G_{\omega}^{3})=s_i(G_\omega)-1$.
\end{Lemma}
\begin{proof}
Let $M$ be an induced matching of  $G_{\omega}^{3}$ with cardinality $\nu(G_{\omega}^{3})$, then  $M\sqcup \{e_i\}$ is an induced matching of  $G_{\omega}$. Hence
$s_i(G_{\omega}) \ge\nu(G_\omega^{3})+1$. On the other hand, if $M'$ is an induced matching of $G_{\omega}$  containing  $e_i$ with cardinality $s_i(G_{\omega})$, then $M'\backslash \{e_i\}$ is an induced matching of  $G_\omega^{3}$. Thus, $\nu(G_\omega^{3}) \ge s_i(G_{\omega})-1$.
\end{proof}

 \begin{Lemma}\label{matching}
If $G$ is a  tree and $e$ is its edge,   then
\[
\nu(G \backslash e)-1 \le \nu(G) \le \nu(G \backslash e)+1.
\]Further, if $\nu(G)=\nu(G\backslash e)-1$ or $\nu(G)=\nu(G\backslash e)+1$, then $s(G)=\nu(G)$,
where $s(G)=\max \{|M|\!: M \text{\ is an induced matching of $G$  containing\ } e\}$.
\end{Lemma}
\begin{proof} Let  $\mathcal{A}$ be the collection of induced matchings of  $G\backslash e$ with cardinality  $\nu(G \backslash e)$ and   $M\in \mathcal{A}$. If $M$ is an induced matching of $G$, then $\nu(G)\ge \nu(G \backslash e)$. Otherwise, there are two edges,  say $e_1,e_2$, in $M$, which incident on two endpoints of $e$. Thus every $M\backslash \{e_i\}$ is an induced matching of $G$, which enforces  $\nu(G)\ge \nu(G \backslash e)-1$. In particular, if this equality holds, then, by the definition of $\nu(G)$, for any  $M\in \mathcal{A}$,  there are two edges $e_1,e_2\in M$ which are incident on two endpoints of $e$. Thus $(M \backslash\{e_1,e_2\}) \sqcup \{e\}$ is an induced matching of  $G$. Hence $s(G) \ge \nu(G)$. Thus $s(G)=\nu(G)$, since $s(G)\le  \nu(G)$ always holds.

Next we prove that $\nu(G) \le \nu(G \backslash e)+1$. Conversely, if $\nu(G) \ge \nu(G \backslash e)+2$, then for any $M\in \mathcal{B}$, where $\mathcal{B}$ is the collection of induced matchings of  $G$ with cardinality  $\nu(G)$, we have $e\in M$.  Indeed, if $e\notin M$, then $M$ is an induced matching of $G\backslash e$, which implies  $\nu(G \backslash e)\ge \nu(G)\ge \nu(G \backslash e)+2$, a contradiction. So
$M\backslash \{e\}$ is an induced matching of $G\backslash e$,  forcing $\nu(G \backslash e)\ge \nu(G)-1\ge \nu(G \backslash e)+1$, a contradiction.
If  $\nu(G)=\nu(G\backslash e)+1$, then, for any induced matching $M$ of  $G$ with cardinality  $\nu(G)$, we have  $e\in M$. It follows that
 $s(G)\ge \nu(G)$. So $s(G)=\nu(G)$, since $s(G)\le  \nu(G)$ always holds.
\end{proof}

\begin{Lemma}\label{number}
Let $G_\omega$ be a non-trivial weighted integrally closed tree as in Remark \ref{setup}. If   $e_{i+2}\in E({G_{\omega}})$, then $s_{i+2}(G_\omega)=\nu(G_\omega^{1})+s_{i+2}(G_\omega^{2})$.
\end{Lemma}
\begin{proof}
Let $M$ be an induced matching of $G_{\omega}$  containing  $e_{i+2}$ with $|M|=s_{i+2}(G_{\omega})$, then $e_i \notin M$. Thus $M=(M \cap E(G_\omega^{1}))\sqcup (M \cap E(G_\omega^{2}))$, since  $G_\omega \backslash e_i=G_{\omega}^{1}\sqcup G_{\omega}^{2}$. It follows that $|M|=|M \cap E(G_\omega^{1})|+|M \cap E(G_\omega^{2})|$. On the one hand, $M \cap E(G_\omega^{1})$ and $M \cap E(G_\omega^{2})$ are an induced matching of $G_\omega^{1}$ and $G_\omega^{2}$, respectively, and $e_{i+2}\in M \cap E(G_\omega^{2})$.
Thus $\nu(G_{\omega}^{1})\ge |M \cap E(G_\omega^{1})|$ and $s_{i+2}(G_{\omega}^{2})\ge |M \cap E(G_\omega^{2})|$, which forces  $\nu(G_{\omega}^{1})+s_{i+2}(G_{\omega}^{2})\ge |M \cap E(G_\omega^{1})|+|M \cap E(G_\omega^{2})|=|M|=s_{i+2}(G_{\omega})$. On the other hand, if $M_1$ and  $M_2$ are an induced matching  of $G_{\omega}^{1}$ and $G_{\omega}^{2}$  with cardinalities $\nu(G_\omega^{1})$ and $s_{i+2}(G_\omega^{2})$, respectively, and $e_{i+2}\in M_2$, then $M_1\cup M_2$ is an induced matching of $G_\omega$  containing the edge $e_{i+2}$.
Therefore, $s_{i+2}(G_{\omega}) \ge |M_1\cup M_2|=\nu(G_\omega^{1})+s_{i+2}(G_{\omega}^{2})$.
\end{proof}

To fully complete the proof of Theorem \ref{reg}, we  have a small result to show.
\begin{Theorem}\label{k=4}
Let $G_\omega$ be a non-trivial integrally closed tree as in Remark \ref{setup}, and let $P_\omega$ be its longest path  of length $(k-1)$ containing all non-trivial edges. If $k=4$, $V(P_\omega)=\{x_1,\ldots,x_4\}$, $\omega_2=\max\{\omega_t\mid\omega_t=\omega(e_t)\text{\ and }  e_t=x_tx_{t+1}\text{\ for any } 1\le t\le 3\}$,
 then  $\reg(I(G_\omega))=2\omega_2$.
\end{Theorem}

\begin{proof}Let $I=I(G_\omega)$ and $I^{\mathcal{P}}$ be its  polarization, then $I^{\mathcal{P}}=J^{\mathcal{P}}+K^{\mathcal{P}}$ is a Betti splitting by Lemma \ref{spliting}, where $\mathcal{G}(J)=\{x_i^{\omega_i} x_{i+1}^{\omega_i}\}$ and $\mathcal{G}(K)=\mathcal{G}(I) \backslash \mathcal{G}(J)$. It follows from Lemma \ref{spliting} and Lemma \ref{polar} that
 \begin{align}
 \reg(I)&=\reg(I^{\mathcal{P}})=\max\{\reg(J^{\mathcal{P}}), \reg(K^{\mathcal{P}}), \reg(J^{\mathcal{P}} \cap K^{\mathcal{P}})-1\}   \notag \\
 & =\max\{\reg(J), \reg(K), \reg(J \cap K)-1\}, \label{eqn: equality1}
 \end{align}
where  $K=I(G_\omega^{1})+I(G_\omega^{2})$,  $J \cap K=JL$, where $L$ is a prime ideal with minimal generators set $N_G(\{x_i,x_{i+1}\})\backslash \{x_i,x_{i+1}\}$. Note that  $\reg(J)=2\omega_2$, $\reg(J \cap K)-1=\reg(JL)-1=2\omega_2$, and $G_\omega^{1}\cup G_\omega^{2}$ is a disjoint union of two
 trivial star, implying that $\reg(K)=3$. Therefore, $\operatorname{reg}(I)=2\omega_2$ from formula (\ref{eqn: equality1}).
\end{proof}

\begin{Theorem}\label{reg}
Let $G_\omega$ be a non-trivial integrally closed tree as in Remark \ref{setup},  and let $P_\omega$ be its longest path  of length $(k-1)$ containing all non-trivial edges, with $V(P_\omega)=\{x_1,\ldots,x_k\}$, $\omega_i=\max\{\omega_t\mid\omega_t=\omega(e_t)\text{\ and }  e_t=x_tx_{t+1}\text{\ for each } 1\le t\le k-1\}$ where    $\omega_i\ge 2$ and $\omega_i\ge \omega_{i+2}$ if $e_{i+2}\in E({P_{\omega}})$. Then,
\begin{itemize}
\item[(1)] if $e_{i+2} \notin E(G_{\omega})$, then  $\reg(I(G_\omega))=2\omega_i$,
\item[(2)] if $e_{i+2} \in E(G_{\omega})$ and $\omega_{i+2}=1$, then $\reg(I(G_\omega))=\max \{\nu(G_\omega)+1, 2\omega_i+(s_i(G_\omega)-1)\}$,
\item[(3)] if $e_{i+2} \in E(G_{\omega})$ and  $\omega_{i+2}\ge 2$, then  $\reg(I(G_\omega))=\max \{\nu(G_\omega)+1, 2\omega_i+(s_i(G_\omega)-1), 2\omega_{i+2}+(s_{i+2}(G_\omega)-1)\}$.
\end{itemize}
\end{Theorem}
\begin{proof}  If  $k\le 3$, or $k=4$ and $i=2$, then the desired results  are obtained from \cite[Theorem 3.1]{ZDCL} and Theorem \ref{k=4}. In the following, we assume that $k=4$ and $i\neq 2$, or $k\ge 5$. Since $\omega_i\ge 2$ and $G_\omega$ is integrally closed, $G_\omega$ has  at most two edges with non-trivial weights. If $G_\omega$ has  only one edge with non-trivial weights, then
by rearranging the subscripts of elements in $V(P_\omega)$, we can say $i\le \lfloor\frac{k}{2}\rfloor$. Thus we always have
$e_{i+2} \in E(G_{\omega})$. Let $I=I(G_\omega)$ and $I^{\mathcal{P}}$ be its  polarization, then $I^{\mathcal{P}}=J^{\mathcal{P}}+K^{\mathcal{P}}$ is a Betti splitting by Lemma \ref{spliting}, where $\mathcal{G}(J)=\{x_i^{\omega_i} x_{i+1}^{\omega_i}\}$, $\mathcal{G}(K)=\mathcal{G}(I) \backslash \mathcal{G}(J)=\mathcal{G}(I(G_\omega^{1}))+\mathcal{G}(I(G_\omega^{2}))$. Meanwhile,  $J \cap K=J L$, where $L$ is an ideal with minimal generators set $(N_G(A)\backslash A)\sqcup \mathcal{G}(I(G_\omega^{3}))$ with  $A=\{x_i,x_{i+1}\}$ and  $G_{\omega}^{3}=G_\omega \backslash N_G(A)$, and $\reg(J)=2\omega_i$. We consider the following two cases:

(1) If $\omega_{i+2}=1$, then $\reg(K)=\reg(I(G_\omega^{1}))+\reg(I(G_\omega^{2}))-1=\nu(G_\omega^{1})+\nu(G_\omega^{2})+1=\nu(G_\omega \backslash e_i)+1$ and $\reg(J\cap K)=\reg(JL)=2\omega_i+\nu(G_\omega^{3})+1=2\omega_i+s_i(G_\omega)$  by Lemmas \ref{sum}, \ref{treetri} and \ref{G3}. It follows from formula (\ref{eqn: equality1}) and  Lemma \ref{matching} that
\begin{align*}
\reg(I)&=\max\{\reg(J), \reg(K), \reg(J \cap K)-1\}\\
  &=\max\{2\omega_i, \nu(G_\omega \backslash e_i)+1, 2\omega_i+s_i(G_\omega)-1\}\\
  &=\max\{\nu(G_\omega \backslash e_i)+1, 2\omega_i+s_i(G_\omega)-1\}\\&
  = \begin{cases}
              \max \{\nu(G_\omega)+1, 2\omega_i+(s_i(G_\omega)-1)\},& \text{if $\nu(G_\omega \backslash e_i)=\nu(G_\omega)$}\\
              2\omega_i+s_i(G_\omega)-1,& \text{otherwise}
        \end{cases}\\
        &=\max \{\nu(G_\omega)+1, 2\omega_i+(s_i(G_\omega)-1)\}.
\end{align*}

(2)  If  $\omega_{i+2}\ge 2$, then
 \begin{align}
\reg(K)&=\reg(I(G_\omega^{1}))+\reg(I(G_\omega^{2}))-1=\nu(G_\omega^{1})+\reg(I(G_\omega^{2}))  \notag \\
&=\nu(G_\omega^{1})+\max \{\nu(G_\omega^{2})+1, 2\omega_{i+2}+(s_{i+2}(G_\omega^2)-1)\}  \notag \\
&=\max \{\nu(G_\omega \backslash e_i)+1,2\omega_{i+2}+(s_{i+2}(G_\omega)-1)\}\label{eqn: equality2}
\end{align}
where the third  and last equalities follow from the proof of the  above case (1), $G_\omega \backslash e_i=G_\omega^{1}\sqcup G_\omega^{2}$ and Lemma \ref{number}, respectively.  Additionally,
$\reg(J \cap K)=\reg(JL)=\reg(J)+\reg(L)=2\omega_i+(\nu(G_\omega^{3})+1)=2\omega_i+s_i(G_\omega)$ by Lemma \ref{G3}.
 Therefore, applying  formula (\ref{eqn: equality2}) and  Lemma \ref{matching}, we can determine that
 \begin{align*}
\reg(I)&=\max\{\reg(J), \reg(K), \reg(J \cap K)-1\}\\
  &=\max\{2\omega_i, \nu(G_\omega \backslash e_i)+1,2\omega_{i+2}+s_{i+2}(G_\omega)-1, 2\omega_i+s_i(G_\omega)-1\}\\
  &=\max\{ \nu(G_\omega \backslash e_i)+1,2\omega_{i+2}+s_{i+2}(G_\omega)-1, 2\omega_i+s_i(G_\omega)-1\}\\
  &=\max\{ \nu(G_\omega)+1,2\omega_{i+2}+s_{i+2}(G_\omega)-1, 2\omega_i+s_i(G_\omega)-1\}.
\end{align*}
 \end{proof}

 Another immediate consequence of Theorem \ref{reg} is \cite[Theorems 4.8 and  4.9]{DZCL}.
 \begin{Corollary}\label{path}
Let $P_{\omega}$ be a non-trivial integrally closed path with $n$ vertices, and  $\omega=\max\{\omega(e) \mid e \in E(P_{\omega})\}$. If
$n\le 4$, then $\reg(I(P_{\omega}))=2\omega$. If $n\ge 5$, then, by   symmetry  and Corollary \ref{integ},  we can assume that  $\omega_i\ge \omega_{i+2}$ and $\omega_i\ge 2$  for some $i\in [n-3]$,  where $e_i=x_ix_{i+1}$ and $\omega_i=\omega(e_i)$ for all $i \in [n-1]$.   Then
\[
\reg(I(P_{\omega}))=\max\{2\omega_i+\lfloor \frac{i-1}{3} \rfloor+\lfloor \frac{n-(i+1)}{3} \rfloor,2\omega_{i+2}+\lfloor \frac{i-2}{3} \rfloor+\lfloor \frac{n-i}{3} \rfloor\}.
\]
\end{Corollary}
\begin{proof}  If  $k\le 3$, or $k=4$ and $i=2$,  then  $\reg(I(P_{\omega}))=2\omega$ according to \cite[Theorem 3.1]{ZDCL} and Theorem \ref{k=4}.  Now we assume that   $k=4$ and $i\neq 2$, or $k\ge 5$.. In this case, we have $\nu(P_{\omega})=\lfloor \frac{n+1}{3} \rfloor$,
$s_i(P_{\omega})=\lfloor \frac{i-1}{3} \rfloor+\lfloor \frac{n-(i+1)}{3} \rfloor+1$ and $s_{i+2}(P_{\omega})=\lfloor \frac{i+1}{3} \rfloor+\lfloor \frac{n-(i+3)}{3} \rfloor+1$ by Lemmas \ref{G3} and  \ref{number}.
We consider two subcases: (i) If $\omega_{i+2}=1$, then, by Theorem \ref{reg}(2), we get $\reg(I(P_{\omega}))=\max \{\nu(P_{\omega})+1, 2\omega_i+(s_i(P_{\omega})-1)\}=2\omega_i+\lfloor \frac{i-1}{3} \rfloor+\lfloor \frac{n-(i+1)}{3} \rfloor$, since $2\omega_i+\lfloor \frac{i-1}{3} \rfloor+\lfloor \frac{n-(i+1)}{3} \rfloor>\lfloor \frac{n+1}{3} \rfloor+1$; (ii) If
$\omega_{i+2}\ge 2$, then, by Theorem \ref{reg}(3), we can conclude that $\reg(I(P_{\omega}))=\max \{\nu(P_{\omega})+1, 2\omega_i+(s_i(P_{\omega})-1),2\omega_{i+2}+(s_{i+2}(P_{\omega})-1) \}=\max\{2\omega_i+\lfloor \frac{i-1}{3} \rfloor+\lfloor \frac{n-(i+1)}{3} \rfloor,2\omega_{i+2}+\lfloor \frac{i-2}{3} \rfloor+\lfloor \frac{n-i}{3} \rfloor\}$, since $2\omega_i+\lfloor \frac{i-1}{3} \rfloor+\lfloor \frac{n-(i+1)}{3} \rfloor>\lfloor \frac{n+1}{3} \rfloor+1$.
\end{proof}

\medskip

\section{Regularity of powers of the edge ideal of an edge-weighted integrally closed tree }
\label{sec:powertree}

This section presents some linear upper bounds on  the  regularity of powers of the edge ideal of an edge-weighted integrally closed tree.  To support our argument, we use the following Lemma, which is similar to \cite[Lemma 4.10]{ZDCL}. The proof is omitted.

 \begin{Lemma}\label{colon}
   Let $G_\omega$ be a non-trivial integrally closed tree as in Remark \ref{setup}, and let $x$ be its leaf with $N_G(x)=\{y\}$ and $\omega(xy)=1$. Then
  \begin{itemize}
\item[(1)] $(I(G_\omega)^t:xy)=I(G_\omega)^{t-1}$;
\item[(2)] $(I(G_\omega)^t, x)=(I(G_\omega\backslash x)^t, x)$;
 \item[(3)] $((I(G_\omega)^t: x), y)=(I(G_\omega\backslash y)^t, y)$;
 \item[(4)] $(I(G_\omega)^t, xy)=(I(G_\omega\backslash x)^t, xy)$;
 \item[(5)] $(I(G_\omega)^t, y)=(I(G_\omega\backslash y)^t, y)$;
\item[(6)]  if $x'$ is also a leaf of $G_\omega$ such that  $N_G(x')=\{y\}$ and $\omega(x'y)=1$, then $((I(G_\omega)^t,xy):x'y)=(I(G_\omega\backslash x)^{t-1},x)$.
\end{itemize}
 \end{Lemma}

 Let $G_\omega$ be an edge-weighted graph  on the vertex set $\mathbf{X}$.  For any subset $\mathbf{Y}\subset\mathbf{X}$, let  $H_\omega$ be its  induced subgraph on  $\mathbf{Y}$.
Let $\KK[\mathbf{X}]=\KK[x|x\in \mathbf{X}]$ be the polynomial ring  over a field $\KK$, then edge ideals $I(G_\omega)$ and  $I(H_\omega)$ are  monomial ideals
in the polynomial ring $\KK[\mathbf{X}]$ and $\KK[\mathbf{Y}]$,  respectively.

  \begin{Lemma}\label{indsubgraph}
  Let $G_\omega$ be an edge-weighted graph  and   $H_\omega$   its induced subgraph, then
  \[
        \beta_{ij}(\KK[\mathbf{Y}]/I(H_\omega)^{t})\le\beta_{ij}(\KK[\mathbf{X}]/I(G_\omega)^{t})
    \]
    for all $i,j$ and $t\ge 1$. In particular, $\reg(I(H_\omega)^{t})\le \reg(I(G_\omega)^{t})$
    for all $t\ge 1$.
  \end{Lemma}
 \begin{proof}
      First, we show that $I(H_\omega)^{t}=I(G_\omega)^{t}\cap \KK[\mathbf{Y}]$ for all $t\ge 1$.
    Since the natural generators of $I(H_\omega)^{t}$ are automatically contained in $I(G_\omega)^{t}$, we have $I(H_\omega)^{t}\subseteq I(G_\omega)^{t}\cap \KK[\mathbf{Y}]$. For the converse  inclusion, let $g\in I(G_\omega)^{t}\cap \KK[\mathbf{Y}]$. We can write $g$ as a finite sum
    \begin{equation*}
        g=\sum_{\substack{u_{\ell i}\in I(G_\omega),\\
     1\le i\le t}}h_{\ell 1,\ldots,\ell t}u_{\ell 1} \cdots u_{\ell t}
    \end{equation*}
  where each $h_{\ell 1,\ldots,\ell t}\in \KK[\mathbf{X}]$. Now, consider the $\KK$-algebra homomorphism $\pi : \KK[\mathbf{X}]\to \KK[\mathbf{Y}]$ by setting
    \[
        \pi(x)=
        \begin{cases}
            x,& \text{if $x$ is  a variable  in $\mathbf{Y}$,}\\
            0,& \text{otherwise}.
        \end{cases}
    \]
  Thus,
    \[
        \pi(u_{\ell i})=
        \begin{cases}
           u_{\ell i}, & \text{if $u_{\ell i}\in I(H_\omega)$,}\\
            0, & \text{otherwise.}
        \end{cases}
    \]
  Since $g\in \KK[\mathbf{Y}]$, we have $\pi(g)=g$. Therefore, we get
\[
         g=\sum_{\substack{u_{\ell i}\in I(G_\omega),\\
     1\le i\le t}}\pi(h_{\ell 1,\ldots,\ell t})\pi(u_{\ell 1}) \cdots \pi(u_{\ell t})=\sum_{\substack{u_{\ell i}\in I(G_\omega),\\
     1\le i\le t}}\pi(h_{\ell 1,\ldots,\ell t})u_{\ell 1} \cdots u_{\ell t}.
\]
    Thus, $g\in I(H_\omega)^{t}$. This completes our proof for $I(H_\omega)^{t}=I(G_\omega)^{t}\cap \KK[\mathbf{Y}]$.

 Consequently, $\KK[\mathbf{Y}]/I(H_\omega)^{t}$ is a $\KK$-subalgebra of $\KK[\mathbf{X}]/I(G_\omega)^{t}$.
    Let $\bar{\pi}:\KK[\mathbf{X}]/I(G_\omega)^{t}\\ \to \KK[\mathbf{Y}]/I(H_\omega)^{t}$ be the homomorphism induced by $\pi$. Since $\pi(I(G_\omega)^{t})\subseteq I(H_\omega)^{t}$, the map $\bar{\pi}$ is well-defined. Notice that the restriction of $\bar{\pi}$ to $\KK[\mathbf{Y}]/I(H_\omega)^{t}$ is the identity map. Thus, $\bar{\pi}$ is surjective, and
    $\KK[\mathbf{Y}]/I(H_\omega)^{t}$ is an algebra retract of $\KK[\mathbf{X}]/I(G_\omega)^{t}$.
    Now, the expected inequalities follow from \cite[Corollary 2.5]{OHH}.
\end{proof}

For any two  vertices $x,y\in V(G)$, their  \emph{distance} $d_G(x,y)$ is the length of the shortest path in $G$ with $x$ and $y$ as  endpoints; if no such path exists, we set $d_G(x,y)=\infty$.
 \begin{Notation}\label{distance}
Let $G_\omega$ be a non-trivial integrally closed tree as in Remark \ref{setup}, and let $P_\omega$ be its longest path  of length $(k-1)$ containing all non-trivial edges. Suppose   $V(P_\omega)=\{x_1,\ldots,x_k\}$.  Given an $x\in V(G_\omega)$, we set $d(x)=\min \{d_G(x, x_j) |1 \le j \le k\text{\ and\ }x\in V(G_\omega)\}$  and $d=\max \{d(x) |\text{\ for any\ } x \in V(G_\omega)\}$.
\end{Notation}

We first determine the regularity of powers of the edge ideal of a non-trivial integrally closed  caterpillar graph under the condition that $k=4$, $d=1$  and  $\omega_2=\max\{\omega(e) \mid e \in E(P_\omega)\}$.

\begin{Theorem}\label{caterpillar1}
Let $G_\omega$ be a non-trivial integrally closed  caterpillar as in Remark \ref{setup},  and let $P_\omega$ be its  path  of length $(k-1)$ containing all non-trivial edges.  If  $k=4$, $d=1$  and  $\omega_2=\max\{\omega(e) \mid e \in E(P_\omega)\}$, then  $\reg(I(G_\omega)^t) =2\omega_2 t$ for all $t\ge 1$.
\end{Theorem}
\begin{proof} Let $I=I(G_{\omega})$. We will prove the statements by induction on $t$. The case where $t=1$ follows from Theorem  \ref{k=4}. So we can assume that $t \ge 2$. Let  $V(P_\omega)=\{x_1,x_2,x_3,x_4\}$, $N_G(x_2)=\{x_1,x_3\}\cup A$ and $N_G(x_3)=\{x_2,x_4\}\cup B$, where $A=\{y_1,\ldots,y_\ell\}$ and  $B=\{u_1,\ldots,u_m\}$, see the graph displayed in Figure $1$. By convention, $A=\emptyset$ if $\ell=0$, and $B=\emptyset$ if $m=0$.  Since $d=1$, we have $\ell+m\ge 1$.

\vspace{0.5cm}
\begin{figure}[htbp]
	\begin{minipage}{0.5\textwidth}
		\centering
		\begin{tikzpicture}[thick, scale=0.8, every node/.style={scale=0.98}]]
			\draw[solid](0.5,0)--(5.5,0);
			\draw[solid](2,0)--(1.3,1);
			\draw[dotted](2,0)--(1.7,1);
			\draw[dotted](2,0)--(2.3,1);
			\draw[solid](2,0)--(2.7,1);
			\draw[solid](4,0)--(3.3,1);
			\draw[dotted](4,0)--(3.7,1);
			\draw[dotted](4,0)--(4.3,1);
			\draw[solid] (4,0)--(4.7,1);
			
			\draw[dotted](1.7,1.3)--(2.3,1.3);
			\draw[dotted](3.7,1.3)--(4.3,1.3);

			\shade [shading=ball, ball color=black] (0.5,0) circle (.07) node [below] {\scriptsize$x_1$};
			\shade [shading=ball, ball color=black] (2,0) circle (.07) node [below] {\scriptsize$x_2$};
			\shade [shading=ball, ball color=black] (4,0) circle (.07) node [below] {\scriptsize$x_3$};
			\shade [shading=ball, ball color=black] (5.5,0) circle (.07) node [below] {\scriptsize$x_4$};
			\shade [shading=ball, ball color=black] (1.3,1) circle (.07) node [above] {\scriptsize$y_1$};
			\shade [shading=ball, ball color=black] (2.7,1) circle (.07) node [above] {\scriptsize$y_\ell$};
			\shade [shading=ball, ball color=black] (3.3,1) circle (.07) node [above] {\scriptsize$u_1$};
			\shade [shading=ball, ball color=black] (4.7,1) circle (.07) node [above] {\scriptsize$u_m$};
		\end{tikzpicture}
	\end{minipage}\hfill
\caption{$Caterpillar\ graph\ with\ k=4, d=1, \omega_2\ge 2$}
\end{figure}
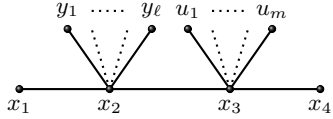
Suppose $\gamma\!:=\ell+m$. We prove the statements by induction
on $\gamma$. Since $G_\omega$ is a non-trivially weighted graph and $\omega_2=\max\{\omega(e) \mid e \in E(P_\omega)\}$, we obtain that
$\omega_2 \ge 2$ and $\omega_1=\omega_3=1$ by  Lemma \ref{integral}. Consider the exact sequences:
\begin{gather*}
\begin{matrix}
0 &\rightarrow& \frac{S}{I^t: u_m}(-1) &\stackrel{\cdot u_m}{\longrightarrow}& \frac{S}{I^t} &\rightarrow& \frac{S}{(I^t, u_m)} &\rightarrow& 0, \\
0 &\rightarrow& \frac{S}{(I^t: u_m):x_3}(-1) &\stackrel{\cdot x_3}{\longrightarrow}& \frac{S}{I^t: u_m} &\rightarrow& \frac{S}{((I^t: u_m),x_3)} &\rightarrow& 0.
\end{matrix}
\end{gather*}		
Note that $(I^t: u_m): x_3=I^{t-1}$, $((I^t: u_m),x_3)=(I(G_\omega \backslash x_3)^t, x_3)$ and $(I^t, u_m)=(I(G_\omega \backslash u_m)^t, u_m)$ by Lemma \ref{colon}, it follows  that
 $\reg((I^t: u_m): x_3)=\reg(I^{t-1})= 2\omega_2(t-1)$ and $\reg(((I^t: u_m), x_3))=\reg((I(G_\omega \backslash x_3)^t,x_3))=\reg(I(G_\omega \backslash x_3)^t)=2t$ by the inductive hypothesis on $t$
 and  \cite[Theorem 3.3]{ZDCL}.

 If $\gamma=1$, then by symmetry we can assume that $m=1,\ell=0$. In this case, we have  $\reg((I^t, u_1))=\reg((I(G_\omega \backslash u_1)^t, u_1))=\reg(I(G_\omega \backslash u_1)^t)=2\omega_2(t-1)+\reg(I(G_\omega \backslash u_1))=2\omega_2(t-1)+\reg(I(P_\omega))=2\omega_2t$ by \cite[Theorem 4.16]{ZDCL} and Corollary \ref{path}. The desired results follow from Lemma \ref{reglemma} and the above exact sequences with $m=1$.

 If $\gamma\ge 2$, then $\reg((I^t, u_m))=\reg((I(G_\omega \backslash u_m)^t,u_m))=\reg(I(G_\omega \backslash u_m)^t)=2\omega_2(t-1)+\reg(I(G_\omega \backslash u_m))=2\omega_2 t$ by the inductive hypothesis on  $\gamma$ and Theorem \ref{reg}. Therefore, the expected results  follow from Lemma \ref{reglemma} and the above exact sequences.
\end{proof}

\begin{Lemma}\label{deletepoint}
Let $G_\omega$ be a non-trivial integrally closed tree as in Remark \ref{setup}, and let $P_\omega$ be its longest path  of length $(k-1)$ containing all non-trivial edges. If $k\ge 5$ or $d\ge 2$, then there exists a trivial  induced  star subgraph $H_\omega$. Let $V(H_\omega)=\{x,y\}\cup C$, where   $C=\{z_1\ldots,z_p\}$  with $p\ge 1$, and $E(H_\omega)=\{xy,xz_1,\ldots,xz_p\}$ with $\deg_{G}(y)\ge 2$ and $\deg_{G}(z_i)=1$ for each $i\in [p]$. If $N_G(x)=\{y\}\cup C$ and $N_G(y)=\{x\} \cup D$, where $D=\{w_1,\ldots,w_q\}$ with $q\ge 1$. Let $J=I(G_\omega \backslash \{x,y\})$, then
$\reg(J) \le \reg(I(G_\omega))-1$.
\end{Lemma}
\begin{proof}	
Let  $\mathcal{A}$ be the collection of induced matchings of  $G_\omega$ with cardinality  $\nu(G_\omega)$. For any $M\in \mathcal{A}$, there exists an edge  $e$ which is taken from the set $\{xy,xz_1,\ldots,xz_p,\\
yw_1,\ldots,yw_q\}$. Thus  $M \backslash \{e\}$ is an induced matching of the graph $G_\omega \backslash \{x,y\}$. It follows that $\nu(G_\omega \backslash \{x,y\}) \ge \nu(G_\omega)-1$. Furthermore, we can see that $\nu(G_\omega \backslash \{x,y\})=\nu(G_\omega)-1$. Indeed, if $\nu(G_\omega \backslash \{x,y\})>\nu(G_\omega)-1$, then there exists an  induced matching $N$ of $G_\omega \backslash \{x,y\}$ such that $|N|>\nu(G_\omega)-1$.
 It is easy to check that $N \cup \{xz_1\}$ is an induced matching of  $G_\omega$. So $\nu(G_\omega)\ge |N|+1>\nu(G_\omega)$, a contradiction.

Similarly, we can show that $s_i(G_\omega \backslash \{x,y\})=s_i(G_\omega)-1$ if $e_i \in E(G_\omega \backslash \{x,y\})$ and $s_{i+2}(G_\omega \backslash \{x,y\})=s_{i+2}(G_\omega)-1$ if $e_{i+2} \in E(G_\omega \backslash \{x,y\})$. Therefore it follows from Theorem \ref{reg} that $\reg(J)=\reg(I(G_\omega))-1$ if $e_i, e_{i+2}\in E(G_\omega \backslash \{x,y\})$. On the other hand, if $e_i\notin E(G_\omega \backslash \{x,y\})$ or $ e_{i+2}\notin E(G_\omega \backslash \{x,y\})$,  then we can deduce that  $\reg(J) \le \reg(I(G_\omega))-1$ by  adopting the same technique.
\end{proof}

The following remark is often used in the following sections.
\begin{Remark}\label{notation}
Let $G_\omega$ be a non-trivial integrally closed tree as in Remark \ref{setup}, and let $P_\omega$ be its longest path  of length $(k-1)$ containing all non-trivial edges. If $k\ge 4$, then there exists an induced  trivial  star subgraph $H_\omega$. Let $V(H_\omega)=\{x,y\}\cup C$,  where $C=\{z_1\ldots,z_p\}$  with $p\ge 1$,  and $E(H_\omega)=\{xy,xz_1,\ldots,xz_p\}$ with $\deg_{G}(y)\ge 2$ and $\deg_{G}(z_i)=1$ for each $i\in [p]$. If $N_G(x)=\{y\}\cup C$ and $N_G(y)=\{x\} \cup D$, where $D=\{w_1,\ldots,w_q\}$ with $q\ge 1$.  Let $L_{1}=(I^t: xz_1), T_{1}=(I^t, xz_1), L_j=(T_{j-1}: xz_j), T_j=(T_{j-1}, xz_j)$ with $2 \le j \le p$, then,  by repeatedly applying Lemma \ref{colon}, we can conclude that
$L_{1}=I^{t-1}$, $T_{1}=(I(G_\omega \backslash z_1)^t,xz_1)$, $L_2=(T_{1}: xz_2)=(I(G_\omega \backslash z_1)^{t-1},z_1)$, $T_{2}=(T_{1},xz_2)=(I(G_\omega \backslash \{z_1,z_2\})^t,xz_1,xz_2)$,  and   $L_j=(T_{j-1}: xz_j)=((I(G_\omega \backslash\{z_1, \ldots, z_{j-1}\})^t, xz_1, \ldots, xz_{j-1}):xz_j)=(I(G_\omega \backslash\{z_1, \ldots, z_{j-1}\})^{t-1},z_1,\\
 \ldots, z_{j-1})$, $T_j=(T_{j-1},xz_j)=(I^t,xz_1,\ldots, xz_j)=(I(G_\omega \backslash\{z_1,\ldots,z_j\})^t,xz_1,\ldots, xz_j)$ for any $3\le j\le p$. At the same time, we also have
 $(T_p:x)=((I(G_\omega \backslash C)^t:x),z_1, \ldots, z_p)$, $(T_p,x)=(I(G_\omega \backslash x)^t,x)$, $((T_p:x):y)=(I(G_\omega \backslash C)^{t-1},z_1,\ldots, z_p)$ and $((T_p:x),y)=((I(G_\omega \backslash C)^t:x),z_1, \ldots, z_p,y)=(I(G_\omega \backslash\{x,y\})^t,z_1,\ldots,z_p,y)$. It follows from Lemma \ref{sum} that $\reg(L_1)=\reg(I^{t-1})$, $\reg(L_j)=\reg(I(G_\omega \backslash\{z_1, \ldots, z_{j-1}\})^{t-1})$ for any $2\le j\le p$, $\reg((T_p,x))=\reg(I(G_\omega \backslash x)^t)$, $\reg(((T_p:x):y))=\reg(I(G_\omega \backslash C)^{t-1})$ and
$\reg(((T_p:x),y))=
	\begin{cases}
		\reg((z_1, \ldots, z_p,y)),& \text{if $k=4,d=1$,}\\
		\reg(I(G_\omega \backslash\{x,y\})^t),& \text{otherwise}.\\
	\end{cases}$
\end{Remark}

Next, we compute the regularity of powers of the edge ideal of a non-trivial integrally closed  caterpillar graph under the condition that $k=4$, $d=1$  and  $\omega_1=\max\{\omega(e) \mid e \in E(P_\omega)\}$.
\begin{Theorem}\label{caterpillar2}
Let $G_\omega$ be a non-trivial integrally closed  caterpillar as in Remark \ref{setup},  and let $P_\omega$ be  its spine of length $(k-1)$ containing all non-trivial edges.  If  $k=4$, $d=1$  and  $\omega_1=\max\{\omega(e) \mid e \in E(P_\omega)\}$, then  $\reg(I(G_\omega)^t) =2\omega_1 t$ for all $t\ge 1$.
\end{Theorem}
\begin{proof}
Let $I=I(G_{\omega})$. We  prove the statements by induction on $t$. The case where $t=1$  follows from Theorem \ref{reg}. So we can assume that $t \ge 2$. Let  $V(P_\omega)=\{x_1,x_2,x_3,x_4\}$, $N_G(x_2)=\{x_1,x_3\}\cup A$ and $N_G(x_3)=\{x_2,x_4\}\cup B$, where $A=\{y_1,\ldots,y_\ell\}$, $B=\{u_1,\ldots,u_m\}$, and $A=\emptyset$ if $\ell=0$, and $B=\emptyset$ if $m=0$ by convention, see Figure $2$ for the case $\omega_3\ge 2$.  Since $d=1$, we have $\ell+m\ge 1$.
Set $\gamma\!:=\ell+m$. Since $G_\omega$ is a non-trivially weighted graph and $\omega_1=\max\{\omega(e) \mid e \in E(P_\omega)\}$, we have
$\omega_1 \ge 2$ and  $\omega_2=1$  by  Lemma \ref{integral}. We distinguish between the following two cases:

(1) If $\omega_3=1$, then there exists an induced trivial star subgraph $H_\omega$ with $V(H_\omega)=\{x_3,x_2\}\cup C$, where $N_G(x_3)=\{x_2\}\cup C$ and $C=\{x_4\}\cup B$. We rearrange the elements in $C$ and assume that $C=\{z_1,\ldots,z_p\}$ with $z_1=x_4$ and $p=m+1$,  see Figure $2$ for the case $\omega_3=1$.

\begin{figure}[htbp]
	\begin{minipage}{0.48\textwidth}
		\centering
		\begin{tikzpicture}[thick, scale=0.8, every node/.style={scale=0.98}]]
			\draw[solid](0.5,0)--(5.5,0);
			\draw[solid](2,0)--(1.3,1);
			\draw[dotted](2,0)--(1.7,1);
			\draw[dotted](2,0)--(2.3,1);
			\draw[solid](2,0)--(2.7,1);
			\draw[solid](4,0)--(3.3,-1);
			\draw[dotted](4,0)--(3.7,-1);
			\draw[dotted](4,0)--(4.3,-1);
			\draw[solid] (4,0)--(4.7,-1);
			
			\draw[dotted](1.7,1.3)--(2.3,1.3);
			\draw[dotted](3.7,-1.3)--(4.3,-1.3);

			\shade [shading=ball, ball color=black] (0.5,0) circle (.07) node [below] {\scriptsize$x_1$};
			\shade [shading=ball, ball color=black] (2,0) circle (.07) node [below] {\scriptsize$x_2$};
			\shade [shading=ball, ball color=black] (4,0) circle (.07) node [above] {\scriptsize$x_3$};
			\shade [shading=ball, ball color=black] (5.5,0) circle (.07) node [below] {\scriptsize$x_4$};
			\shade [shading=ball, ball color=black] (1.3,1) circle (.07) node [above] {\scriptsize$y_1$};
			\shade [shading=ball, ball color=black] (2.7,1) circle (.07) node [above] {\scriptsize$y_\ell$};
			\shade [shading=ball, ball color=black] (3.3,-1) circle (.07) node [below] {\scriptsize$u_1$};
			\shade [shading=ball, ball color=black] (4.7,-1) circle (.07) node [below] {\scriptsize$u_m$};
		\end{tikzpicture}
	\subcaption*{The  case $\omega_3\ge 2$}
	\end{minipage}\hfill
	\begin{minipage}{0.48\textwidth}
		\centering
		\begin{tikzpicture}[thick, scale=0.8, every node/.style={scale=0.98}]]
			\draw[solid](0.5,0)--(5.5,0);
			\draw[solid](2,0)--(1.3,1);
			\draw[dotted](2,0)--(1.7,1);
			\draw[dotted](2,0)--(2.3,1);
			\draw[solid](2,0)--(2.7,1);
			\draw[solid](4,0)--(2.8,-1);
			\draw[dotted](4,0)--(3.7,-1);
			\draw[dotted](4,0)--(4.3,-1);
			\draw[solid] (4,0)--(5.2,-1);
			
			\draw[dotted](1.7,1.3)--(2.3,1.3);
			\draw[dotted](3.7,-1.3)--(4.3,-1.3);

			\shade [shading=ball, ball color=black] (0.5,0) circle (.07) node [below] {\scriptsize$x_1$};
			\shade [shading=ball, ball color=black] (2,0) circle (.07) node [below] {\scriptsize$x_2=y$};
			\shade [shading=ball, ball color=black] (4,0) circle (.07) node [above] {\scriptsize$x_3=x$};
			\shade [shading=ball, ball color=black] (5.5,0) circle (.07) node [below] {\scriptsize$x_4=z_1$};
			\shade [shading=ball, ball color=black] (1.3,1) circle (.07) node [above] {\scriptsize$y_1$};
			\shade [shading=ball, ball color=black] (2.7,1) circle (.07) node [above] {\scriptsize$y_\ell$};
			\shade [shading=ball, ball color=black] (2.8,-1) circle (.07) node [below] {\scriptsize$u_1=z_2$};
			\shade [shading=ball, ball color=black] (5.2,-1) circle (.07) node [below] {\scriptsize$u_m=z_p$};
		\end{tikzpicture}
	\subcaption*{The  case $\omega_3=1$}
	\end{minipage}\hfill
\caption{$Caterpillar\ graph\ with\  k=4,d=1,\omega_1\ge 2$}
\end{figure}
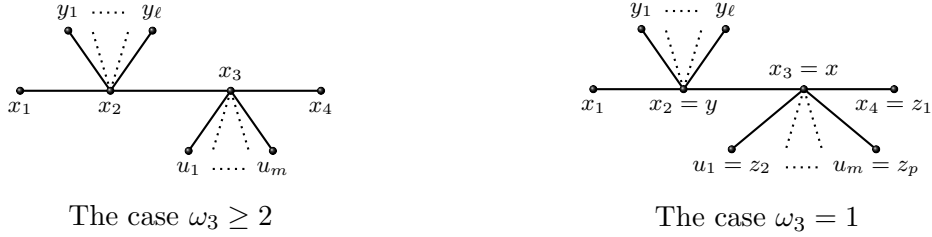
\hspace{-0.5cm}Adopting the notations in Remark \ref{notation} and replacing $x$ and $y$ by $x_3$ and $x_2$, respectively, we can obtain from Theorem \ref{reg}(2), \cite[Theorems 3.1 and 3.3]{ZDCL} and  the inductive hypothesis on $t$ that
$\reg(L_1)=\reg(I^{t-1})=2\omega_1 (t-1)$, $\reg(L_j)=\reg(I(G_\omega \backslash\{z_1, \ldots, z_{j-1}\})^{t-1})=2\omega_1 (t-1)$ for any $2\le j\le p$, $\reg((T_p,x))=\reg(I(G_\omega \backslash x)^t)=2\omega_1 t$, $\reg(((T_p:x):y))=\reg(I(G_\omega \backslash C)^{t-1})=2\omega_1 (t-1)$ and
$\reg(((T_p:x),y))=1$, since $k=4$ and  $d=1$. Thus, the expected results follow from Lemma \ref{reglemma} and  the following exact sequences
	\begin{gather}\label{eqn: equality3}
		\begin{matrix}
			0 &\rightarrow& \frac{S}{L_1}(-2) &\stackrel{\cdot xz_1}{\longrightarrow}& \frac{S}{I^t} &\rightarrow& \frac{S}{T_1} &\rightarrow& 0, \\
			0 &\rightarrow& \frac{S}{L_2}(-2) &\stackrel{\cdot xz_2}{\longrightarrow}& \frac{S}{T_1} &\rightarrow& \frac{S}{T_2} &\rightarrow& 0, \\
			&  &\vdots&  &\vdots&  &\vdots&  &\\
			0 &\rightarrow& \frac{S}{L_p}(-2) &\stackrel{\cdot xz_p}{\longrightarrow}& \frac{S}{T_{p-1}} &\rightarrow& \frac{S}{T_p} &\rightarrow& 0, \\
			0 &\rightarrow& \frac{S}{T_p:x}(-1) &\stackrel{\cdot x}{\longrightarrow}& \frac{S}{T_p} &\rightarrow& \frac{S}{(T_p,x)} &\rightarrow& 0, \\
			0 &\rightarrow& \frac{S}{(T_p:x):y}(-1) &\stackrel{\cdot y}{\longrightarrow}& \frac{S}{T_p:x} &\rightarrow& \frac{S}{((T_p:x),y)} &\rightarrow& 0. \\
		\end{matrix}
	\end{gather}

(2) Suppose $\omega_3\ge 2$. Let $J=(I^t,x_2y_1,\ldots,x_2y_\ell,x_3u_1,\ldots,x_3u_m)$, then $J=(I(G_\omega \backslash (A\cup B))^t,x_2y_1,\ldots,x_2y_\ell,x_3u_1,\ldots,x_3u_m)$, where $G_\omega \backslash (A\cup B)$ is the path  $P_\omega$ of length $3$. It follows from  \cite[Lemma 4.11]{ZDCL} that
\begin{align*}
		J:(x_2x_3)^{t-1}&=(I(G_\omega \backslash (A\cup B)),y_1,\ldots,y_\ell,u_1,\ldots,u_m),\\
		(J,x_2x_3)&=((x_1x_2)^{t\omega_1},x_2x_3,(x_3x_4)^{t\omega_3},x_2y_1,\ldots,x_2y_\ell,x_3u_1,\ldots,x_3u_m),\\
		((J:(x_2x_3)^i),x_2x_3)&=((x_1x_2)^{(t-i)\omega_1},x_2x_3,(x_3x_4)^{(t-i)\omega_3},y_1,\ldots,y_\ell,u_1,\ldots,u_m)
	\end{align*}
for any $i \in [t-2]$. Therefore, by Lemma \ref{sum},  Lemma \ref{reg} and \cite[Theorem 4.8]{ZDCL}, we can deduce that $\reg(J:(x_2x_3)^{t-1})=\reg(I(G_\omega \backslash (A\cup B)))=2\omega_1$ and $\reg(((J:(x_2x_3)^i),x_2x_3))=2\omega_1(t-i)$ for all $i=0,1,\ldots,t-2$. Therefore, $\reg(J)=2\omega_1 t$ by Lemma \ref{reglemma} and the following exact sequences
	\begin{gather}\label{eqn: equality4}
		\begin{matrix}
			0 &\rightarrow& \frac{S}{J:x_2x_3}(-2) &\stackrel{\cdot x_2x_3}{\longrightarrow}& \frac{S}{J} &\rightarrow& \frac{S}{(J,x_2x_3)} &\rightarrow& 0, \\
			0 &\rightarrow& \frac{S}{J:(x_2x_3)^2}(-2) &\stackrel{\cdot x_2x_3}{\longrightarrow}& \frac{S}{J:x_2x_3} &\rightarrow& \frac{S}{((J:x_2x_3),x_2x_3)} &\rightarrow& 0, \\
			&  &\vdots&  &\vdots&  &\vdots&  &\\
			0 &\rightarrow& \frac{S}{J:(x_2x_3)^{t-2}}(-2) &\stackrel{\cdot x_2x_3}{\longrightarrow}& \frac{S}{J:(x_2x_3)^{t-3}} &\rightarrow& \frac{S}{((J:(x_2x_3)^{t-3}),x_2x_3)} &\rightarrow& 0, \\
			0 &\rightarrow& \frac{S}{J:(x_2x_3)^{t-1}}(-2) &\stackrel{\cdot x_2x_3}{\longrightarrow}& \frac{S}{J:(x_2x_3)^{t-2}} &\rightarrow& \frac{S}{((J:(x_2x_3)^{t-2}),x_2x_3)} &\rightarrow& 0. \\
		\end{matrix}
	\end{gather}
In order to compute $\reg(I^t)$, we distinguish into the following three cases:

(i) If $\ell, m\ge 1$, then let $F_1=(I^t:x_3u_1),K_1=(I^t,x_3u_1),F_j=(K_{j-1}:x_3u_j),K_j=(K_{j-1},x_3u_j)$ for all $2\le j \le m$, $H_1=(K_m:x_2y_1),L_1=(K_m,x_2y_1), H_s=(L_{s-1},x_2y_s),L_s=(L_{s-1},x_2y_s)$ for all $2\le s\le \ell$, then by Lemma \ref{colon}, we have
 \begin{align*}
 F_1&=I^{t-1},\ \ \ K_1=(I^t,x_3u_1)=(I(G_\omega \backslash u_1)^t,x_3u_1),\\
 F_j&=(K_{j-1}:x_3u_j)=(I(G_\omega \backslash \{u_1,\ldots,u_{j-1}\})^{t-1},u_1,\ldots,u_{j-1}),\\
 K_j&=(K_{j-1},x_3u_j)=(I(G_\omega \backslash \{u_1,\ldots,u_j\})^t,x_3u_1,\ldots,x_3u_j),\\	
 H_1&=(K_m:x_2y_1)=(I(G_\omega \backslash \{u_1,\ldots,u_m\})^{t-1},x_3u_1,\ldots,x_3u_m),\\
L_1&=(F_m,x_2y_1)=(I(G_\omega \backslash \{u_1,\ldots,u_m,y_1\})^t,x_3u_1,\ldots,x_3u_m,x_2y_1),\\
 H_s&=(L_{s-1}:x_2y_s)=(I(G_\omega \backslash \{u_1,\ldots,u_m,y_1,\ldots,y_{s-1}\})^{t-1},x_3u_1,\ldots,x_3u_m,y_1,\ldots,y_{s-1}),\\
L_s&=(L_{s-1},x_2y_s)=(I(G_\omega \backslash \{u_1,\ldots,u_m,y_1,\ldots,y_s\})^t,x_3u_1,\ldots,x_3u_m,x_2y_1,\ldots,x_2y_s).
 \end{align*}
Thus we have  $\reg(F_j)=2\omega_1(t-1)$ for all $1\le j\le m$, and  $\reg(L_\ell)=\reg((I^t,x_3u_1,\ldots,\\
x_3u_m,x_2y_1,\ldots,x_2y_\ell))=2\omega_1 t$ by  similar arguments as in the proof of  the regularity of $J$ above.

Now, we compute  $\reg(H_1)$. For any $1\le j\le m-1$, by Lemma \ref{colon}, we have
\begin{align*}
H_1:u_1&=(I(G_\omega \backslash x_3)^{t-1},x_3),\\
(H_1,u_1,\ldots,u_j)&=(I(G_\omega \backslash \{u_1,\ldots,u_m\})^{t-1},u_1,\ldots,u_j,x_3u_{j+1},\ldots,x_3u_m),\\
(H_1,u_1,\ldots,u_j):u_{j+1}&=(I(G_\omega \backslash x_3)^{t-1},u_1,\ldots,u_j,x_3),\\
(H_1,u_1,\ldots,u_m)&=(I(G_\omega \backslash \{u_1,\ldots,u_m\})^{t-1},u_1,\ldots,u_m).
\end{align*}
It follows from Theorem \ref{reg} and the inductive hypothesis on $t$  that $\reg(H_1:u_1)=\reg(I(G_\omega \backslash x_3)^{t-1})=2\omega_1 (t-1)$, $\reg((H_1,u_1,\ldots,u_m))
=\reg(I(G_\omega \backslash \{u_1,\ldots,u_m\})^{t-1})=2\omega_1 (t-1)$, $\reg((H_1,u_1,\ldots,u_j):u_{j+1})=\reg(I(G_\omega \backslash x_3)^{t-1})=2\omega_1 (t-1)$. Thus, by Lemma \ref{reglemma} and the following exact sequences
	\begin{gather}\label{eqn: equality5}
		\begin{matrix}
			0 &\rightarrow& \frac{S}{H_1:u_1}(-1) &\stackrel{\cdot u_1}{\longrightarrow}& \frac{S}{H_1} &\rightarrow& \frac{S}{(H_1,u_1)} &\rightarrow& 0, \\
			0 &\rightarrow& \frac{S}{(H_1,u_1):u_2}(-1) &\stackrel{\cdot u_2}{\longrightarrow}& \frac{S}{(H_1,u_1)} &\rightarrow& \frac{S}{(H_1,u_1,u_2)} &\rightarrow& 0, \\
			&  &\vdots&  &\vdots&  &\vdots&  &\\
			0 &\rightarrow& \frac{S}{(H_1,u_1,\ldots,u_{m-2}):u_{m-1}}(-1) &\stackrel{\cdot u_{m-1}}{\longrightarrow}& \frac{S}{(H_1,u_1,\ldots,u_{m-2})} &\rightarrow& \frac{S}{(H_1,u_1,\ldots,u_{m-1})} &\rightarrow& 0, \\
			0 &\rightarrow& \frac{S}{(H_1,u_1,\ldots,u_{m-1}):u_m}(-1) &\stackrel{\cdot u_m}{\longrightarrow}& \frac{S}{(H_1,u_1,\ldots,u_{m-1})} &\rightarrow& \frac{S}{(H_1,u_1,\ldots,u_m)} &\rightarrow& 0, \\
		\end{matrix}
	\end{gather}
	we obtain   $\reg(H_1) \le 2\omega_1 (t-1)+1$. Similarly, we also can  deduce that  $\reg(H_s) \le 2\omega_1 (t-1)+1$ for any $1\le s\le \ell$. Hence, by Lemma \ref{reglemma} and the following  exact sequences
\begin{gather}\label{eqn: equality6}
		\begin{matrix}
			0 &\rightarrow& \frac{S}{F_1}(-2) &\stackrel{\cdot x_3u_1}{\longrightarrow}& \frac{S}{I^t} &\rightarrow& \frac{S}{K_1} &\rightarrow& 0, \\
			0 &\rightarrow& \frac{S}{F_2}(-2) &\stackrel{\cdot x_3u_2}{\longrightarrow}& \frac{S}{K_1} &\rightarrow& \frac{S}{K_2} &\rightarrow& 0, \\
			&  &\vdots&  &\vdots&  &\vdots&  &\\
			0 &\rightarrow& \frac{S}{F_{m-1}}(-2) &\stackrel{\cdot x_3u_{m-1}}{\longrightarrow}& \frac{S}{K_{m-2}} &\rightarrow& \frac{S}{K_{m-1}} &\rightarrow& 0, \\
			0 &\rightarrow& \frac{S}{F_m}(-2) &\stackrel{\cdot x_3u_m}{\longrightarrow}& \frac{S}{K_{m-1}} &\rightarrow& \frac{S}{K_m} &\rightarrow& 0,
	\end{matrix}
	\end{gather}
and
\begin{gather}\label{eqn: equality7}
		\begin{matrix}
			0 &\rightarrow& \frac{S}{H_1}(-2) &\stackrel{\cdot x_2y_1}{\longrightarrow}& \frac{S}{K_m} &\rightarrow& \frac{S}{L_1} &\rightarrow& 0, \\
			0 &\rightarrow& \frac{S}{H_2}(-2) &\stackrel{\cdot x_2y_2}{\longrightarrow}& \frac{S}{L_1} &\rightarrow& \frac{S}{L_2} &\rightarrow& 0, \\
			&  &\vdots&  &\vdots&  &\vdots&  &\\
			0 &\rightarrow& \frac{S}{H_{\ell-1}}(-2) &\stackrel{\cdot x_3y_{\ell-1}}{\longrightarrow}& \frac{S}{L_{\ell-2}} &\rightarrow& \frac{S}{L_{\ell-1}} &\rightarrow& 0, \\
			0 &\rightarrow& \frac{S}{H_\ell}(-2) &\stackrel{\cdot x_3y_\ell}{\longrightarrow}& \frac{S}{L_{\ell-1}} &\rightarrow& \frac{S}{L_\ell} &\rightarrow& 0, \\
		\end{matrix}
	\end{gather}
 we can obtain $\reg(I^t) = 2\omega_1 t$.

(ii) If $\ell=0$, then  $\reg(K_m)=\reg((I^t,x_3u_1,\ldots, x_3u_m))=2\omega_1 t$ and $\reg(F_j)=2\omega_1(t-1)$ for all $1\le j\le m$ as shown in the proof above. It follows from  by Lemma \ref{reglemma} and the   exact sequences (\ref{eqn: equality6}) above that $\reg(I^t) = 2\omega_1 t$.

(iii) If $m=0$, then  $\reg(L_\ell)=\reg((I^t,x_2y_1,\ldots, x_2y_\ell))=2\omega_1 t$ and  $\reg(H_j)=2\omega_1 (t-1)$ for all $1\le j\le \ell$ as shown in the proof above. It follows from  by Lemma \ref{reglemma} and the  exact sequences (\ref{eqn: equality7})above  that $\reg(I^t) = 2\omega_1 t$.
\end{proof}

In the following, we compute the regularity of powers of the edge ideal of a non-trivial integrally closed  caterpillar graph under the condition that $k=5$, or $k=6$ and $i=2$, and  $d=1$.
\begin{Theorem}\label{k=5,k=6}
Assuming $d=1$, let $G_\omega$ be a non-trivial integrally closed  caterpillar, as described in Remark \ref{setup},  and let $P_\omega$ be  the longest path of length $(k-1)$ containing all non-trivial edges. Also, assume that    $\omega_i=\max\{\omega(e) \mid e \in E(P_\omega)\}$ with $\omega_i\ge 2$ and $\omega_i\ge \omega_{i+2}$ if $e_{i+2}\in E({P_{\omega}})$. If $k=5$, or $k=6$ and $i=2$, then $\reg(I(G_\omega)^t))\le 2\omega_i(t-1)+\reg(I(G_\omega))$ for all $t\ge 1$.
\end{Theorem}
\begin{proof}
Let $I=I(G_\omega)$ and $V(G_\omega)=V(P_\omega)\cup \{y_1,\ldots,y_q\}$, where $V(P_\omega)=\{x_1,\ldots,x_k\}$, then it follows from Theorem \ref{reg} that
\[
\reg(I)=\begin{cases}
		2\omega_1+1,& \text{if $k=5,i=1$,}\\
		2\omega_2,& \text{if $k=5,i=2$ and $\omega_2>\omega_4$,}\\
	2\omega_2+1,& \text{if $k=5,i=2$ and $\omega_2=\omega_4$, or $k=6,i=2$.}
\end{cases}
\]
We prove the statements by induction on $t$. The case where $t = 1$ is trivial. Thus, we may assume that $t\ge 2$. If $k=5$, then $i=1$ or $i=2$ due to symmetry.
 We consider the following three cases:

 (1) If $k=5$ and $i=1$, then, by Lemmas \ref{sum} and  \ref{colon}, \cite[Theorems 3.1 and 3.3]{ZDCL} and  the inductive hypothesis on $t$, we obtain that  $\reg(((I^t:x_5):x_{4}))=\reg(I^{t-1})\le 2\omega_1(t-1)+1$ and  $\reg(((I^t:x_5),x_{4}))=\reg((I(G_\omega \backslash x_{4})^t,x_{4}))=\reg(I(G_\omega \backslash x_{4})^t)=2\omega_1t$.

The following two subcases are being considered. (i) If $q=1$, then,
by some simple calculations, we obtain
 \[
 \reg((I^t,x_5))=\reg(I(G_\omega \backslash x_5)^t)=\begin{cases}
 	2\omega_1t+1,& \text{if $N_{G}(y_1)=\{x_4\}$,}\\
 	2\omega_1t,& \text{ otherwise.}\\
 \end{cases}
 \]
Replacing $x_5$ and $x_4$ for $x_u$ and $x_v$ respectively  in the following short exact sequences,
\begin{gather}\label{eqn: equality8}
		\begin{matrix}
			0 &\rightarrow& \frac{S}{I^t:x_u}(-1) &\stackrel{\cdot x_u}{\longrightarrow}& \frac{S}{I^t} &\rightarrow& \frac{S}{(I^t,x_u)} &\rightarrow& 0, \\
			0 &\rightarrow& \frac{S}{(I^t:x_u):x_{v}}(-1) &\stackrel{\cdot x_{v}}{\longrightarrow}& \frac{S}{I^t:x_u} &\rightarrow& \frac{S}{((I^t:x_u),x_{v})} &\rightarrow& 0, \\
		\end{matrix}
	\end{gather}
we obtain from  Lemma \ref{reglemma} that  $\reg(I^t)\le  2\omega_1t+1$.

(ii) If $q\ge 2$, then, by Lemmas \ref{sum}, \ref{colon}, \ref{indsubgraph} and \ref{caterpillar2} and the inductive hypothesis on $q$, we get  $\reg((I^t,x_5))=\reg(I(G_\omega \backslash x_5)^t)\le 2\omega_1(t-1)+\reg(I(G_\omega \backslash x_5))\le 2\omega_1t+1$.
Again using the  exact sequences (\ref{eqn: equality8}), replacing $x_5$ and $x_4$ for $x_u$ and $x_v$  respectively, we  have  $\reg(I^t)\le  2\omega_1t+1$.

 (2) If $k=5$ and $i=2$, it is sufficient to assume  $\omega_2>\omega_4$, since the case where $\omega_2=\omega_4$ follows from  case (1) above.
  In this case,   by Lemmas \ref{sum} and \ref{colon}, \cite[Theorems 3.1 and 3.3]{ZDCL} and the inductive hypothesis on $t$, we can deduce that  $\reg(((I^t:x_1):x_{2}))=\reg(I^{t-1})\le 2\omega_2(t-1)$, $\reg(((I^t:x_1),x_{2}))=\reg((I(G_\omega \backslash x_{2})^t,x_{2}))=\reg(I(G_\omega \backslash x_{2})^t)=2\omega_4t$.

If $q=1$, then $\reg((I^t,x_1))=\reg(I(G_\omega \backslash x_1)^t)=2\omega_2t$ by Lemmas \ref{sum}, \ref{colon} and \ref{caterpillar2}, and \cite[Theorem 4.16]{ZDCL}. Therefore,  replacing $x_5$ and $x_4$ for $x_u$ and $x_v$,  respectively,   in the  exact sequences (\ref{eqn: equality8}),  it can be concluded that  $\reg(I^t)\le 2\omega_2t$.
If $q\ge 2$, then, $\reg((I^t,x_1))=\reg(I(G_\omega \backslash x_1)^t)\le 2\omega_2(t-1)+\reg(I(G_\omega \backslash x_1))=2\omega_2t$ by Lemmas \ref{sum} and \ref{colon}, Theorem \ref{caterpillar2}, and the inductive hypothesis on $q$. Substituting $x_u$ and $x_v$ for $x_1$ and $x_2$ in the  exact sequences (\ref{eqn: equality8}), we can deduce
 $\reg(I^t) \le  2\omega_2 t$.

 (3) If $k=6$ and $i=2$,  it follows from Lemmas \ref{sum} and \ref{colon}, Theorem \ref{caterpillar1},  case (2) above, and the inductive hypothesis on $t$ that  $\reg(((I^t:x_6):x_{5}))=\reg(I^{t-1})\le 2\omega_2(t-1)+1$ and $\reg(((I^t:x_6),x_{5}))=\reg(I(G_\omega \backslash x_{5})^t)\le 2\omega_2t$.

If $q=1$, then $\reg((I^t,x_6)) =\reg(I(G_\omega \backslash x_6)^t)\le 2\omega_2t+1$ by Lemmas \ref{sum}, \ref{colon} and \ref{indsubgraph}, and \cite[Theorem 4.16]{ZDCL} and  cases (1) and  (2) above.
  Therefore, by replacing $x_u$ and $x_v$ for $x_6$ and $x_5$ in the  exact sequences (\ref{eqn: equality8}),  we  can obtain  $\reg(I^t)\le 2\omega_2t+1$.
If $q\ge 2$,  it follows from Lemmas \ref{sum},  \ref{colon} and \ref{indsubgraph},   cases (1) and (2) above, and the inductive hypothesis on $q$ that $\reg((I^t,x_6))=\reg(I(G_\omega \backslash x_6)^t)\le 2\omega_2(t-1)+\reg(I(G_\omega \backslash x_6))\le 2\omega_2t+1$.
  Again, substituting $x_u$ and $x_v$ for $x_6$ and $x_5$ in the  exact sequences (\ref{eqn: equality8}),  we can conclude that $\reg(I^t) \le  2\omega_2 t+1$.
This completes  the proof.
\end{proof}

Next, we prove a main result of this section.
\begin{Theorem}\label{tree,d=1}
	Assuming $d=1$, let $G_\omega$ be a non-trivial integrally closed  caterpillar, as described in Remark \ref{setup},  and let $P_\omega$ be  the longest  spine of length $(k-1)$ containing all non-trivial edges. Also, assume that    $\omega_i=\max\{\omega(e) \mid e \in E(P_\omega)\}$ with $\omega_i\ge 2$ and $\omega_i\ge \omega_{i+2}$ if $e_{i+2}\in E({P_{\omega}})$.  Then $\reg(I(G_\omega)^t)\le 2\omega_i(t-1)+\reg(I(G_\omega))$ for all $t\ge 1$.
\end{Theorem}
\begin{proof}
Let $I=I(G_\omega)$ and   $V(P_\omega)=\{x_1,\ldots,x_k\}$. We prove the statements by induction on $t$. The case where $t = 1$ is trivial. So we can assume that $t\ge 2$. By Remark \ref{notation}, there exists an induced trivial star subgraph $H_\omega$ with $V(H_\omega)=\{x,y,z_1\ldots,z_p\}$, where $p\ge 1$ and $N_G(x)=\{y,z_1\ldots,z_p\}$, and $E(H_\omega)=\{xy,xz_1,\ldots,xz_p\}$ with $\deg_{G}(y)\ge 2$ and $\deg_{G}(z_i)=1$ for all $i\in [p]$.  We distinguish between the following two cases:

 (1) If $i=1$. We prove the statements by induction on $k$, and the case where  $k\le 5$ is verified  separately by  Theorem \ref{caterpillar2}, Theorem \ref{k=5,k=6} and \cite[Theorems 3.1 and 3.3]{ZDCL}. Now we can assume that  $k\ge 6$. Adopting the notations in Remark \ref{notation}, we choose $x=x_{k-1}$, $y=x_{k-2}$, $L_{1}=(I^t: xz_1)$, $T_{1}=(I^t, xz_1)$, $L_j=(T_{j-1}: xz_j)$, $T_j=(T_{j-1}, xz_j)$ for all $2 \le j \le p$, then, by Lemma \ref{indsubgraph}, Remark \ref{notation} and  the inductive hypothesis on  $t$, we can conclude that  $\reg(L_1)\le 2\omega_1(t-2)+\reg(I)$, $\reg(((T_p:x):y))\le 2\omega_1(t-2)+\reg(I)$, $\reg(L_j)\le 2\omega_1(t-2)+\reg(I)$ for all $2\le j\le p$. Meanwhile, by  Lemma \ref{deletepoint} and the inductive hypothesis on $k$, we can get $\reg((T_p,x))=\reg(I(G_\omega \backslash x)^t)\le 2\omega_1(t-1)+\reg(I),\reg(((T_p:x),y))=\reg(I(G_\omega \backslash\{x,y\})^t)\le 2\omega_1(t-1)+\reg(I(G_\omega \backslash\{x,y\})) \le 2\omega_1(t-1)+\reg(I)-1$.
 Therefore, we obtain from that Lemma \ref{reglemma} and the short exact sequences (\ref{eqn: equality3}) that $\reg(I^t)\le 2\omega_1(t-1)+\reg(I)$.

 (2) If $i=2$ and  $k\le 6$, then the desired results follow from \cite[Theorems 3.1 and 3.3]{ZDCL}, Theorems \ref{caterpillar1} and \ref{k=5,k=6}. If $i=2$ and  $k\ge 7$, we choose $x=x_{k-1},y=x_{k-2}$, see  Figure $3$  for the  case $\omega_1\ge 2$.  Similarly,  if $i\ge 3$, we choose $x=x_2,y=x_3$, see  Figure $3$ for the  case $\omega_3\ge 2$. In these two cases, the desired results can be shown by similar arguments as in
case (1) above, so we omit its proof.
\end{proof}

\begin{figure}[htbp]
	\begin{minipage}{0.48\textwidth}
		\centering
		\begin{tikzpicture}[thick, scale=0.8, every node/.style={scale=0.98}]]
			\draw[solid](0.5,0)--(8.5,0);
			\draw[solid](1.5,0)--(1.2,1);
			\draw[dotted] (1.5,0)--(1.8,1);
			\draw[dotted] (3,0)--(2.7,1);
			\draw[dotted] (3,0)--(3.3,1);
			\draw[dotted] (4.5,0)--(4.2,1);
			\draw[dotted] (4.5,0)--(4.8,1);
			\draw[dotted] (6,0)--(5.7,1);
			\draw[dotted] (6,0)--(6.3,1);
			\draw[dotted] (7.5,0)--(7.2,1);
			\draw[solid](7.5,0)--(7.8,1);
			
			\draw[dotted] (4,1.3)--(5,1.3);

		\shade [shading=ball, ball color=black] (0.5,0) circle (.07) node [below] {\scriptsize$x_1$};
		\shade [shading=ball, ball color=black] (1.5,0) circle (.07) node [below] {\scriptsize$x_2$};
		\shade [shading=ball, ball color=black] (3,0) circle (.07) node [below] {\scriptsize$x_3$};
		\shade [shading=ball, ball color=black] (4.5,0) circle (.07) node [below] {\scriptsize$x_4$};
		\shade [shading=ball, ball color=black] (6,0) circle (.07) node [below] {\scriptsize$x_5=y$};
		\shade [shading=ball, ball color=black] (7.5,0) circle (.07) node [below] {\scriptsize$x_6=x$};
		\shade [shading=ball, ball color=black] (8.5,0) circle (.07) node [above] {\scriptsize$x_7=z_1$};
		\shade [shading=ball, ball color=black] (1.2,1) circle (.07) node [above] {\scriptsize$y_1$};
		\shade [shading=ball, ball color=black] (7.8,1) circle (.07) node [above] {\scriptsize$y_q=z_p$};
		\end{tikzpicture}
	\subcaption*{The case $\omega_1\ge 2$}
	\end{minipage}\hfill
	\begin{minipage}{0.48\textwidth}
		\centering
		\begin{tikzpicture}[thick, scale=0.8, every node/.style={scale=0.98}]]
			\draw[solid](0.5,0)--(8.5,0);
			\draw[solid](1.5,0)--(1.2,1);
			\draw[dotted] (1.5,0)--(1.8,1);
			\draw[dotted] (3,0)--(2.7,1);
			\draw[dotted] (3,0)--(3.3,1);
			\draw[dotted] (4.5,0)--(4.2,1);
			\draw[dotted] (4.5,0)--(4.8,1);
			\draw[dotted] (6,0)--(5.7,1);
			\draw[dotted] (6,0)--(6.3,1);
			\draw[dotted] (7.5,0)--(7.2,1);
			\draw[solid](7.5,0)--(7.8,1);
			
			\draw[dotted] (4,1.3)--(5,1.3);

			\shade [shading=ball, ball color=black] (0.5,0) circle (.07) node [above] {\scriptsize$x_1=z_1$};
			\shade [shading=ball, ball color=black] (1.5,0) circle (.07) node [below] {\scriptsize$x_2=x$};
			\shade [shading=ball, ball color=black] (3,0) circle (.07) node [below] {\scriptsize$x_3=y$};
			\shade [shading=ball, ball color=black] (4.5,0) circle (.07) node [below] {\scriptsize$x_4$};
			\shade [shading=ball, ball color=black] (6,0) circle (.07) node [below] {\scriptsize$x_5$};
			\shade [shading=ball, ball color=black] (7.5,0) circle (.07) node [below] {\scriptsize$x_6$};
			\shade [shading=ball, ball color=black] (8.5,0) circle (.07) node [above] {\scriptsize$x_7$};
			\shade [shading=ball, ball color=black] (1.2,1) circle (.07) node [above] {\scriptsize$y_1=z_p$};
			\shade [shading=ball, ball color=black] (7.8,1) circle (.07) node [above] {\scriptsize$y_q$};
		\end{tikzpicture}
	\subcaption*{The case $\omega_3\ge 2$}
	\end{minipage}\hfill
	\caption{ $Caterpillar\ graph\ with\ k=7,d=1$}
\end{figure}
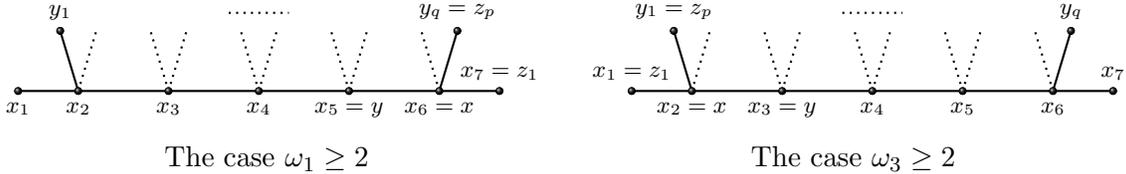

Now, we provide some linear upper bounds on the regularity of powers of the edge ideal of a non-trivial integrally closed tree  under the condition that $k=4$, $d=2$  and  $\omega_1=\max\{\omega(e) \mid e \in E(P_\omega)\}$.
\begin{Theorem}\label{k=4,d=2}
Let $G_\omega$ be a non-trivial integrally closed tree as in Remark \ref{setup},  and let $P_\omega$ be its longest path of length $(k-1)$ containing all non-trivial edges.  If $k=4$, $d=2$ and  $\omega_1=\max\{\omega(e) \mid e \in E(P_\omega)\}$, then  $\reg(I(G_\omega)^t) \le 2\omega_1(t-1)+\reg(I(G_\omega))$ for all $t\ge 1$.
\end{Theorem}
\begin{proof}
Let $I=I(G_\omega)$,  $V(P_\omega)=\{x_1,\ldots,x_k\}$,  $d(x)=\min \{d_G(x, x_j) |1 \le j \le k\text{\ and\ }x\in V(G_\omega)\}$ and  $\gamma=|\{x \in V(G_\omega)\mid d(x)=2\}|$.
We prove the statements by induction on $t$. The case where $t = 1$ is trivial. So we can assume that $t\ge 2$. Since $d=2$, there exists an induced trivial star subgraph $H_\omega$ with $V(H_\omega)=\{x,y\}\cup C$,  where $d(x)=1$, $N_G(x)=\{y\}\cup C$, $C=\{z_1\ldots,z_p\}$  with $p\ge 1$,  and $E(H_\omega)=\{xy,xz_1,\ldots,xz_p\}$ with $\deg_{G}(y)\ge 2$ and $\deg_{G}(z_i)=1$ for all $i\in [p]$, see  Figure $4$ for the  case $k=4$ and $d=2$.
From the definition of $P_\omega$, we know that if $\omega_3=1$, then  $y=x_2$, and if $\omega_3\ge 2$, then $y=x_2$ or $y=x_3$.  For the sake of consistency, we will always set $y=x_2$, since the case where $y=x_3$ can be shown by similar arguments. Adapting the notations in Remark \ref{notation}, let $L_{1}=(I^t: xz_1), T_{1}=(I^t, xz_1), L_j=(T_{j-1}: xz_j), T_j=(T_{j-1}, xz_j)$ for any $2 \le j \le p$, then, by  Lemma \ref{indsubgraph},   the inductive hypothesis on  $t$, we have $\reg(L_1)\le 2\omega_1(t-2)+\reg(I)$, $\reg(((T_p:x):y))\le 2\omega_1(t-2)+\reg(I)$ and $\reg(L_j)\le 2\omega_1(t-2)+\reg(I)$ for all $2\le j\le p$.

To compute $\reg(I^t)$ from the exact sequences (\ref{eqn: equality3}), we have to compute $\reg((T_p,x))$ and $\reg(((T_p:x),y))$.
From Remark \ref{notation}, Theorem \ref{caterpillar2} and \cite[Theorem  4.16]{ZDCL}, we can deduce that if $\gamma=1$, then $\reg((T_p,x))=\reg(I(G_\omega \backslash x)^t)=2\omega_1 t$.
If $\gamma\ge 2$, then $\reg((T_p,x))=\reg(I(G_\omega \backslash x)^t)\le 2\omega_1 (t-1)+\reg(I(G_\omega \backslash x))\le 2\omega_1 (t-1)+\reg(I)$ again by Lemma \ref{indsubgraph} and the inductive hypothesis on $\gamma$.
From  Remark \ref{notation}, Lemmas \ref{sum}, \ref{treetri} and \ref{deletepoint} and \cite[Theorems 3.1, 3.3 and 4.16]{ZDCL}, we can deduce that $\reg(((T_p:x),y))=\reg(I(G_\omega \backslash \{x,y\})^t)=2\omega_3t$ when $\gamma=1$, and $\reg(((T_p:x),y))=2(t-1)+\reg(I(G_\omega \backslash \{x,y\}))\le 2\omega_1(t-1)+\reg(I)-1$ if $\gamma\ge 2$ and  $\omega_3=1$, and  $\reg(((T_p:x),y))\le 2\omega_3(t-1)+\reg(I(G_\omega \backslash \{x,y\}))\le 2\omega_1(t-1)+\reg(I)-1$ if $\omega_3\ge 2$ and $\gamma\ge 2$,
again by the case $\omega_3=1$ above   and Theorem \ref{caterpillar2}.

In any case, we  can always obtain $\reg(I^t) \le 2\omega_1(t-1)+\reg(I)$ by Lemma \ref{reglemma} and the   exact sequences (\ref{eqn: equality3}).
 \end{proof}

\begin{figure}[htbp]
	\begin{minipage}{0.48\textwidth}
		\centering
		\begin{tikzpicture}[thick, scale=0.8, every node/.style={scale=0.98}]]
			\draw[solid](0.5,0)--(5.3,0);
			\draw[solid](2,0)--(0.7,1);
			\draw[solid](2,0)--(1.7,1);
			\draw[solid](2,0)--(2.7,1);
			\draw[solid](4,0)--(3.1,1);
			\draw[solid](4,0)--(3.7,1);
			\draw[solid](4,0)--(4.7,1);
			\draw[solid](1.7,1)--(1,2);
			\draw[dotted](1.7,1)--(1.3,2);
			\draw[dotted](1.7,1)--(1.6,2);
			\draw[solid](1.7,1)--(2,2);
			\draw[dotted](3.7,1)--(3.3,2);
			\draw[dotted](3.7,1)--(4,2);
			\draw[dotted](4.7,1)--(4.4,2);
			\draw[dotted](4.7,1)--(5.1,2);
			
			\draw[dotted](1.3,2.3)--(1.6,2.3);

			\shade [shading=ball, ball color=black] (0.5,0) circle (.07) node [below] {\scriptsize$x_1$};
			\shade [shading=ball, ball color=black] (2,0) circle (.07) node [below] {\scriptsize$x_2=y$};
			\shade [shading=ball, ball color=black] (4,0) circle (.07) node [below] {\scriptsize$x_3$};
			\shade [shading=ball, ball color=black] (5.3,0) circle (.07) node [below] {\scriptsize$x_4$};
			\shade [shading=ball, ball color=black] (0.7,1) circle (.07);
			\shade [shading=ball, ball color=black] (1.7,1) circle (.07) node [left] {\scriptsize$x$};
			\shade [shading=ball, ball color=black] (2.7,1) circle (.07);
			\shade [shading=ball, ball color=black] (3.1,1) circle (.07);
			\shade [shading=ball, ball color=black] (3.7,1) circle (.07);
			\shade [shading=ball, ball color=black] (4.7,1) circle (.07);
			\shade [shading=ball, ball color=black] (1,2) circle (.07) node [above] {\scriptsize$z_1$};
			\shade [shading=ball, ball color=black] (2,2) circle (.07) node [above] {\scriptsize$z_p$};
		\end{tikzpicture}
	\subcaption*{The case $k=4,d=2$}
	\end{minipage}\hfill
	\begin{minipage}{0.48\textwidth}
		\centering
		\begin{tikzpicture}[thick, scale=0.8, every node/.style={scale=0.98}]]
			\draw[solid](0.5,0)--(7.3,0);
			\draw[solid](2,0)--(0.7,1);
			\draw[solid](2,0)--(1.7,1);
			\draw[solid](2,0)--(2.7,1);
			\draw[solid](4,0)--(3.1,1);
			\draw[solid](4,0)--(3.7,1);
			\draw[solid](4,0)--(4.7,1);
			\draw[solid](6,0)--(5.5,1);
			\draw[solid](6,0)--(6.7,1);
			\draw[solid](1.7,1)--(1,2);
			\draw[solid](1.7,1)--(2,2);
			\draw[solid](2.7,1)--(2.5,2);
			\draw[solid](2.7,1)--(3.3,2);
			\draw[solid](4.7,1)--(4.2,2);
			\draw[solid](2,2)--(1.5,3);
			\draw[solid](2,2)--(2.5,3);
			\draw[dotted](2,2)--(1.8,3);
			\draw[dotted](2,2)--(2.2,3);
			\draw[dotted](3.3,2)--(2.9,3);
			\draw[dotted](3.3,2)--(3.6,3);
			\draw[dotted](4.2,2)--(3.8,3);
			\draw[dotted](4.2,2)--(4.6,3);
		
		    \draw[dotted](1.8,3.3)--(2.2,3.3);

			\shade [shading=ball, ball color=black] (0.5,0) circle (.07) node [below] {\scriptsize$x_1$};
			\shade [shading=ball, ball color=black] (2,0) circle (.07) node [below] {\scriptsize$x_2=y$};
			\shade [shading=ball, ball color=black] (4,0) circle (.07) node [below] {\scriptsize$x_3$};
			\shade [shading=ball, ball color=black] (6,0) circle (.07) node [below] {\scriptsize$x_4$};
			\shade [shading=ball, ball color=black] (7.3,0) circle (.07) node [below] {\scriptsize$x_5$};
			\shade [shading=ball, ball color=black] (0.7,1) circle (.07);
			\shade [shading=ball, ball color=black] (1.7,1) circle (.07) node [left] {\scriptsize$y$};
			\shade [shading=ball, ball color=black] (2.7,1) circle (.07);
			\shade [shading=ball, ball color=black] (3.1,1) circle (.07);
			\shade [shading=ball, ball color=black] (3.7,1) circle (.07);
			\shade [shading=ball, ball color=black] (4.7,1) circle (.07);
			\shade [shading=ball, ball color=black] (5.5,1) circle (.07);
			\shade [shading=ball, ball color=black] (6.7,1) circle (.07);
	    	\shade [shading=ball, ball color=black] (1,2) circle (.07);
			\shade [shading=ball, ball color=black] (2,2) circle (.07) node [left] {\scriptsize$x$};
			\shade [shading=ball, ball color=black] (2.5,2) circle (.07);
			\shade [shading=ball, ball color=black] (3.3,2) circle (.07);
			\shade [shading=ball, ball color=black] (4.2,2) circle (.07);
			\shade [shading=ball, ball color=black] (1.5,3) circle (.07) node [above] {\scriptsize$z_1$};
			\shade [shading=ball, ball color=black] (2.5,3) circle (.07) node [above] {\scriptsize$z_p$};
		\end{tikzpicture}
	\subcaption*{The case $k=5,d=3$}
	\end{minipage}\hfill
	\caption{$Caterpillar\  graph$}
	\label{Fig:H}
\end{figure}
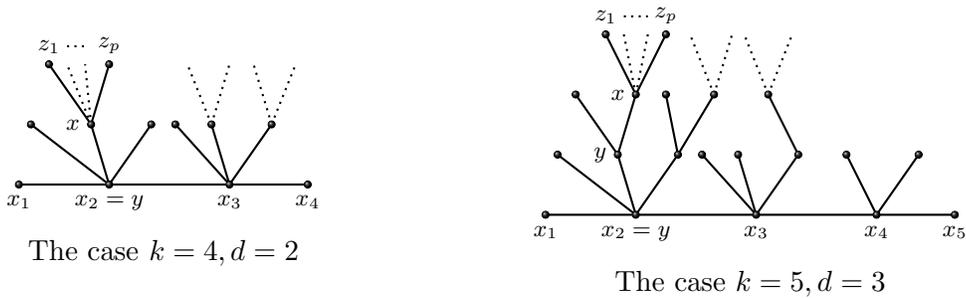
Next, we prove another main result of this section.
\begin{Theorem}\label{tree,d=2}
Assuming  $d=2$, let $G_\omega$ be a non-trivial integrally closed  tree, as described in Remark \ref{setup},  and let $P_\omega$ be  the longest  path of length $(k-1)$ containing all non-trivial edges. Also, assume that    $\omega_i=\max\{\omega(e) \mid e \in E(P_\omega)\}$  with $\omega_i\ge 2$ and $\omega_i\ge \omega_{i+2}$ if $e_{i+2}\in E({P_{\omega}})$.  Then $\reg(I(G_\omega)^t)\le 2\omega_i(t-1)+\reg(I(G_\omega))$ for all $t\ge 1$.
\end{Theorem}
\begin{proof}
Let $I=I(G_\omega)$,  $V(P_\omega)=\{x_1,\ldots,x_k\}$,  $d(x)=\min \{d_G(x, x_j) |1 \le j \le k\text{\ and\ }x\in V(G_\omega)\}$,  and  $\gamma=|\{x \in V(G_\omega)\mid d(x)=2\}|$.
We prove the statements by induction on $t$ and $k$. The case where $t = 1$ is trivial and where  $t\ge 2$ and $k=4$ is verified by Theorem \ref{k=4,d=2}. So we can assume that $t\ge 2$ and $k\ge 5$.
As in the proof of Theorem \ref{k=4,d=2}, there exists an induced trivial star subgraph $H_\omega$ with $V(H_\omega)=\{x,y,z_1\ldots,z_p\}$,  where $d(x)=1$, $N_G(x)=\{y,z_1\ldots,z_p\}$ and $p\ge 1$,  and $E(H_\omega)=\{xy,xz_1,\ldots,xz_p\}$ with $\deg_{G}(y)\ge 2$ and $\deg_{G}(z_i)=1$ for all $i\in [p]$. By the definition of $P_\omega$,   $y=x_j$ with $2\le j\le k-1$.
 For consistency,  we choose $y=x_3$, since other $y$ can be shown by similar arguments. Adapting the notations in Remark \ref{notation}, let $L_{1}=(I^t: xz_1), T_{1}=(I^t, xz_1), L_j=(T_{j-1}: xz_j), T_j=(T_{j-1}, xz_j)$ for all $2 \le j \le p$, then, by  Lemma \ref{indsubgraph},   the inductive hypothesis on $t$, we can obtain that  $\reg(L_1)\le 2\omega_i(t-2)+\reg(I)$, $\reg(((T_p:x):y))\le 2\omega_i(t-2)+\reg(I)$ and  $\reg(L_j)\le 2\omega_i(t-2)+\reg(I)$ for all $2\le j\le p$.

If $\gamma=1$, then by Lemmas \ref{sum}, \ref{treetri}, \ref{indsubgraph} and \ref{deletepoint}, Theorem \ref{tree,d=1}, and \cite[Theorems 3.1, 3.3 and 4.16]{ZDCL} and  Lemma \ref{indsubgraph}, we can deduce that $\reg((T_p,x))=\reg(I(G_\omega \backslash x)^t)\le 2\omega_i(t-1)+\reg(I(G_\omega \backslash x))\le 2\omega_i(t-1)+\reg(I)$ and $\reg(((T_p:x),y))=\reg(I(G_\omega \backslash \{x,y\})^t)\\
\le 2\omega_i(t-1)+\reg(I(G_\omega \backslash \{x,y\}))\le 2\omega_i(t-1)+\reg(I)-1$.

If $\gamma\ge 2$, then, by Lemmas \ref{sum},  \ref{treetri},  \ref{indsubgraph} and \ref{deletepoint}, \cite[Theorems 3.1, 3.3 and 4.16]{ZDCL}, and the inductive hypothesis on $\gamma$, we have $\reg((T_p,x))=\reg(I(G_\omega \backslash x)^t)\le 2\omega_i (t-1)+\reg(I(G_\omega \backslash x))\le 2\omega_i (t-1)+\reg(I)$ and $\reg(((T_p:x),y))=\reg(I(G_\omega \backslash \{x,y\})^t)\le 2\omega_i(t-1)+\reg(I)-1$.

Therefore, in every case, we  can always obtain $\reg(I^t) \le 2\omega_i(t-1)+\reg(I)$ by Lemma \ref{reglemma} and the   exact sequences (\ref{eqn: equality3}).
 \end{proof}

It is time for the important result of this section.
\begin{Theorem}\label{tree,d>2}
Let $G_\omega$ be a non-trivial integrally closed tree as in Remark \ref{setup},  and let $P_\omega$ be its longest path of length $(k-1)$ containing all non-trivial edges. If  $\omega=\max\{\omega(e) \mid e \in E(P_\omega)\}$, then $\reg(I(G_\omega)^t) \le 2\omega(t-1)+\reg(I(G_\omega))$ for all $t\ge 1$.
\end{Theorem}
\begin{proof}
Let $I=I(G_\omega)$,  $V(P_\omega)=\{x_1,\ldots,x_k\}$,  $d(x)=\min \{d_G(x, x_j) |1 \le j \le k\text{\ and\ }x\in V(G_\omega)\}$, and  $\gamma=|\{x \in V(G_\omega)\mid d(x)=d\}|$.
	We prove the statements by induction on $t$ and $d$. The case where $t=1$ is trivial and the case where $t\ge 2$ and $d\le 2$ is proved separately
in \cite[Theorem 4.16]{ZDCL}, Theorems \ref{tree,d=1} and \ref{tree,d=2}. Thus, in the following, we can assume that $t\ge 2$ and $d\ge 3$. As in the proof of Theorem \ref{k=4,d=2}, there exists an induced trivial star subgraph $H_\omega$ with $V(H_\omega)=\{x,y,z_1\ldots,z_p\}$,  where  $d(x)=d-1$, $N_G(x)=\{y,z_1\ldots,z_p\}$ and $p\ge 1$,  and $E(H_\omega)=\{xy,xz_1,\ldots,xz_p\}$ with $\deg_{G}(y)\ge 2$ and $\deg_{G}(z_i)=1$ for all $i\in [p]$ (, see  Figure $4$ for the  case $k=5$ and $d=3$). Adapting the notations in Remark \ref{notation}, let $L_{1}=(I^t: xz_1), T_{1}=(I^t, xz_1), L_j=(T_{j-1}: xz_j), T_j=(T_{j-1}, xz_j)$ for all $2 \le j \le p$,
 then, by  Lemma \ref{indsubgraph},   the inductive hypothesis on $t$, we have $\reg(L_1)\le 2\omega(t-2)+\reg(I)$, $\reg(((T_p:x):y))\le 2\omega(t-2)+\reg(I)$, $\reg(L_j)\le 2\omega(t-2)+\reg(I)$ for all $2\le j\le p$.

If $\gamma=1$, then,  by Lemmas \ref{indsubgraph} and  \ref{deletepoint}, and  the inductive hypothesis on $d$, we obtain $\reg((T_p,x))=\reg(I(G_\omega \backslash x)^t)\le 2\omega(t-1)+\reg(I(G_\omega \backslash x))\le 2\omega(t-1)+\reg(I)$ and $\reg(((T_p:x),y))=\reg(I(G_\omega \backslash \{x,y\})^t)\le 2\omega(t-1)+\reg(I(G_\omega \backslash \{x,y\}))\le 2\omega(t-1)+\reg(I)-1$.

If $\gamma\ge 2$, then, by  Lemmas \ref{indsubgraph} and \ref{deletepoint}, and the inductive hypothesis on $\gamma$ and $d$, we can deduce  $\reg((T_p,x))=\reg(I(G_\omega \backslash x)^t)\le 2\omega (t-1)+\reg(I(G_\omega \backslash x))\le 2\omega (t-1)+\reg(I)$ and $\reg(((T_p:x),y))=\reg(I(G_\omega \backslash \{x,y\})^t)\le 2\omega(t-1)+\reg(I(G_\omega \backslash \{x,y\}))\le 2\omega(t-1)+\reg(I)-1$.

Therefore, in each case, we  can always obtain $\reg(I^t) \le 2\omega(t-1)+\reg(I)$ by Lemma \ref{reglemma} and the   exact sequences (\ref{eqn: equality3}).
 \end{proof}

In special cases, the upper bound in Theorem \ref{tree,d>2} can reach.
\begin{Theorem}\label{eq}
	Let $G_\omega$ be a non-trivial integrally closed tree as in Remark \ref{setup},  and let $P_\omega$ be its longest path of length $(k-1)$ containing all non-trivial edges. If  $\omega_i=\max\{\omega(e) \mid e \in E(P_\omega)\}$ and $\reg(I(G_\omega))=2\omega_i+(s_i(G_\omega)-1)$, then $\reg(I(G_\omega)^t) = 2\omega_it+(s_i(G_\omega)-1)$ for all $t\ge 1$.
\end{Theorem}
\begin{proof}
	Let $I=I(G_\omega)$ and $I^t=J^t+K$, where $J=(x_i^{\omega_i}x_{i+1}^{\omega_i})$ and $K=\mathcal{G}(I^t)\backslash \mathcal{G}(J^t)$, and let $(J^t)^{\mathcal{P}}$, $K^{\mathcal{P}}$ and $(I^t)^{\mathcal{P}}$ be the polarizations of $J^t$, $K$ and $I^t$, respectively. Then $(I^t)^{\mathcal{P}}=J^{\mathcal{P}}+K^{\mathcal{P}}$ and $(J^t)^{\mathcal{P}}\cap K^{\mathcal{P}}=(J^t\cap K)^{\mathcal{P}}$ by \cite[Proposition 2.3]{F}.
By Lemmas \ref{spliting} and  \ref{polar}, we know that $(I^t)^{\mathcal{P}}=(J^t)^{\mathcal{P}}+K^{\mathcal{P}}$ is a Betti splitting and that
 \begin{align}
		\reg(I^t)&=\reg((I^t)^{\mathcal{P}})=\max\{\reg((J^t)^{\mathcal{P}}), \reg(K^{\mathcal{P}}), \reg((J^t)^{\mathcal{P}} \cap K^{\mathcal{P}})-1\}   \notag \\
		& =\max\{\reg(J^t), \reg(K), \reg(J^t \cap K)-1\}. \label{eqn: equality9}
	\end{align}
Since $\omega_i=\max\{\omega(e) \mid e \in E(G_\omega)\}$, we have
  \begin{align*}
J^t \cap K&=(x_i^{t\omega_i}x_{i+1}^{t\omega_i})\cap (\mathcal{G}(I^t)\backslash \mathcal{G}(J^t))\\
&=(x_i^{t\omega_i}x_{i+1}^{t\omega_i})\cap [(x_i^{\omega_i}x_{i+1}^{\omega_i})^{t-1}(\mathcal{G}(I)\backslash \mathcal{G}(J))+(x_i^{\omega_i}x_{i+1}^{\omega_i})^{t-2}(\mathcal{G}(I)\backslash \mathcal{G}(J))^2+\cdots\\
&+(x_i^{\omega_i}x_{i+1}^{\omega_i})(\mathcal{G}(I)\backslash \mathcal{G}(J))^{t-1}+(\mathcal{G}(I)\backslash \mathcal{G}(J))^t]\\
&=(x_i^{t\omega_i}x_{i+1}^{t\omega_i})\cap ((x_i^{\omega_i}x_{i+1}^{\omega_i})^{t-1}(\mathcal{G}(I)\backslash \mathcal{G}(J)))\\
&=J^tL
\end{align*}
where $L$ is an ideal, its  minimal generator set is $(N_G(A)\backslash A)\sqcup \mathcal{G}(I(G_\omega^{3}))$ with $A=\{x_i,x_{i+1}\}$ and $G_{\omega}^{3}=G_\omega \backslash N_G(A)$. So
$\reg(J^t\cap K)=\reg(J^tL)=\reg(J^t)+\reg(L)=2\omega_it+\nu(G_\omega^3)+1=2\omega_it+s_i(G_\omega)$ by Lemmas \ref{sum} and \ref{G3}.

On the other hand, let $H$ and $H^{\prime}$ be hypergraphs associated  with $\mathcal{G}((I^t)^{\mathcal{P}})$ and $\mathcal{G}(K^{\mathcal{P}})$, respectively, then $H^{\prime}$ is an induced subhypergraph of $H$. Hence,
	$\reg(K)=\reg(K^{\mathcal{P}}) \le \reg((I^t)^{\mathcal{P}}) \le 2\omega_it+(s_i(G_\omega)-1)$ by Lemma \ref{polar}, Theorem \ref{tree,d>2} and \cite[Lemma 3.1]{H}.
	Therefore, $\reg(I(G_\omega)^t) = 2\omega_it+(s_i(G_\omega)-1)$  from formula (\ref{eqn: equality9}).
\end{proof}

The following examples show that the upper bound in Theorem \ref{tree,d>2} can be strict.
\begin{Example}
Let $G_\omega$ be a non-trivial integrally closed tree  as in Remark \ref{setup}, its edge ideal is   $I(G_\omega)=(x_1x_2,x_2^2x_3^2, x_3x_4,x_4x_5, x_3x_6,x_6x_7,x_3x_8,x_8x_9)$.  Let $P_\omega$ be its longest path  containing all non-trivial edges with $V(P_\omega)=\{x_1,x_2,x_3,x_4,x_5\}$, then $\omega_2=\max\{\omega_t\mid \omega_t=\omega(e_t) \text{ and } e_t=x_tx_{t+1} \text{ for any } 1\le t\le 4\}=2$, $s_2(G_\omega)=1$ and $\nu(G_\omega)=4$. Thus $\reg(I(G_\omega))=\nu(G_\omega)+1$ by Theorem \ref{reg}. By using CoCoA, we obtain $\reg(I(G_\omega)^2)=8<2\omega_2+\reg(I(G_\omega))$.
\end{Example}

\begin{Example}
Let $G_\omega$ be a non-trivial integrally closed tree  as in Remark \ref{setup}, its edge ideal is  $I(G_\omega)=(x_1^3x_2^3,x_2x_3,x_3^2x_4^2,x_2x_5, x_5x_6,x_2x_7,x_7x_8,x_2x_9,x_9x_{10})$.  Let $P_\omega$ be its longest path  containing all non-trivial edges with $V(P_\omega)=\{x_1,x_2,x_3,x_4\}$, then $\omega_1=\max\{\omega_t\mid \omega_t=\omega(e_t) \text{ and } e_t=x_tx_{t+1} \text{ for any } 1\le t\le 4\}=3$, $\omega_3=2$, $s_1(G_\omega)=1$, $s_3(G_\omega)=3$  and $\nu(G_\omega)=4$. Thus $\reg(I(G_\omega))=2\omega_{3}+(s_{3}(G_\omega)-1)=7$ by Theorem \ref{reg}. By using CoCoA, we obtain $\reg(I(G_\omega)^2)=12<2\omega_1+\reg(I(G_\omega))$.
\end{Example}

\medskip
\hspace{-6mm} {\bf Acknowledgments}

 \vspace{3mm}
\hspace{-6mm}  This research is supported by the Natural Science Foundation of Jiangsu Province (No. BK20221353). The authors are grateful to the software systems \cite{Co} and \cite{GS}
 for providing us with a large number of examples.





\end{document}